\documentclass[10pt]{amsart}
\usepackage{amsmath}
\usepackage{amsthm}
\usepackage{amsfonts}
\usepackage{mathrsfs}
\usepackage{amssymb}
\usepackage{amscd}
\usepackage[all]{xy}
\usepackage{pgf}
\usepackage{tikz}
\usetikzlibrary{cd}

\usepackage{amscd}

\usepackage{enumerate}
\usepackage{hyperref}

    \newtheorem{Lem}{Lemma}[section]
    \newtheorem{Lem-Def}{Lemma-Definition}[section]
    \newtheorem{Prop}[Lem]{Proposition}

    \newtheorem{Thm}[Lem]{Theorem}  
    
		\newtheorem*{Thm*}{Theorem}

\theoremstyle{definition}

    \newtheorem{Def}[Lem]{Definition}
    \newtheorem{Exa}[Lem]{Example}
    \newtheorem{Rem}[Lem]{Remark}

\newcommand{\ra}{\rightarrow}

\newcommand{\E}{\mathcal E}

\newcommand{\EL}{\mathcal L}
\newcommand{\T}{\mathcal T}
\newcommand{\col}{\colon}
\newcommand{\ora}{\overrightarrow}
\newcommand{\ol}{\overline}
\newcommand{\wh}{\widehat}

\newcommand{\D}{\mathcal{D}}
\newcommand{\PS}{\mathbf{PSD}}
\newcommand{\ps}{\mathcal{PSD}}
\newcommand{\PD}{\mathbf{PD}}
\newcommand{\pd}{\mathcal{PD}}

\newcommand{\sst}{\mathcal{SSD}}

\newcommand{\qs}{\mathcal{QD}}
\newcommand{\J}{\mathcal{J}}

\newcommand{\R}{\mathbb{R}}

\newcommand{\Ima}{\textnormal{Im}}
\newcommand{\trop}{\textnormal{trop}}
\newcommand{\Aut}{\textnormal{Aut}}

\DeclareMathOperator{\rk}{rk}
\DeclareMathOperator{\val}{val}
\DeclareMathOperator{\Div}{Div}
\DeclareMathOperator{\pol}{pol}
\DeclareMathOperator{\pr}{pr}
\DeclareMathOperator{\Id}{Id}
\DeclareMathOperator{\semi}{ss}
\DeclareMathOperator{\allsemi}{s}
\DeclareMathOperator{\Prin}{Prin}
\DeclareMathOperator{\ord}{ord}
\DeclareMathOperator{\Pic}{Pic}
\DeclareMathOperator{\cone}{cone}

\begin{document}

\newcommand\encircle[1]{
  \tikz[baseline=(X.base)] 
    \node (X) [draw, shape=circle, inner sep=0] {\strut #1};}

\setcounter{tocdepth}{1}

 \title[A universal tropical Jacobian over $M_g^{\trop}$]{A universal tropical Jacobian over $M_g^{\trop}$}

\author{Alex Abreu, Sally Andria, Marco Pacini, and Danny Taboada}

\email{alexbra1@gmail.com, sally.andrya@gmail.com, pacini.uff@gmail.com, and  danny\textunderscore daft@hotmail.com}

\maketitle

\begin{abstract}
We introduce and study polystable divisors on a tropical curve, which are the tropical analogue of  polystable torsion-free rank-1 sheaves on a nodal curve. We construct a universal tropical Jacobian over the moduli space of tropical curves of genus $g$. This space parametrizes equivalence classes of tropical curves of genus $g$ together with a $\mu$-polystable divisor, and can be seen as a tropical counterpart of  Caporaso universal Picard scheme. We describe polyhedral decompositions of the Jacobian of a tropical curve via polystable divisors, relating them with other known polyhedral decompositions.
\end{abstract}

\bigskip



\tableofcontents

\section{Introduction}

\subsection{Background}

  In the last few years, tropical geometry has uncovered deep connections with algebraic geometry, in particular when applied to the moduli theory of algebraic curves. This interplay became evident in  paper \cite{ACP}, which contains a beautiful interpretation of the moduli space of tropical curves as the skeleton of the Berkovich analytification of the moduli space of stable curves. In this sense, tropical geometry can be seen as a confluence point between algebraic and analytic geometry. The result motivated the construction and the analysis of tropical analogues of other interesting moduli spaces in algebraic geometry (see \cite{AP1}, \cite{BR}, \cite{CMP}, \cite{CMR}, \cite{MUW}, \cite{R1}, \cite{R2}, \cite{U}, just to make few examples). This new point of view could also provide a better understanding of the boundary of a given compactification. 
 
 Sometimes, the construction of a tropical moduli space can also help to handle the combinatorial issues of a given problem in algebraic geometry, or even suggest its solution. For example, a ``tropically inspired" resolution of the universal Abel map was recently given in \cite{H} and \cite{AP2} (though the tropical setting does not appear explicitly in \cite{H}).
 
A fundamental problem in algebraic geometry is how to attain a compactified universal Picard variety over the moduli space of stable curves. There are essentially two proper compactifications, Caporaso universal Picard scheme $\ol{P}_{d,g}$, constructed in \cite{C}, and later generalized in \cite{Pand} by Pandharipande for  higher ranks, and the universal Picard stack $\ol{\J}_{\mu,g}$, constructed in  \cite{M15} by Melo building on the work of Esteves \cite{Es01}. The space $\ol{P}_{d,g}$ lives over the moduli space of stable curve, in stark contrast with the space $\ol{\J}_{\mu,g}$, that naturally lives only over the moduli of \emph{pointed} curves, as it is necessary to have enough sections to define its objects. On the other hand, the second one is a fine moduli stack, while the first one   is not.

A natural problem in tropical geometry is the construction of a universal tropical Jacobian over the moduli space $M_g^{\trop}$ of tropical curve (as usual, in the category of generalized cone complexes). This problem has attracted a lot of interest in the last years, due to the centrality of its algebro-geometric counterpart. To our knowledge, the first construction of a universal tropical Jacobian (as a topologial space) for tropical curves of a fixed combinatorial type was carried out in \cite{Len}. The first and third author recently introduced a universal tropical Jacobian in \cite{AP1}. This space is the tropical counterpart of the Picard stack $\ol{\J}_{\mu,g}$. As expected, the construction holds over the moduli space of \emph{pointed} tropical curves, and could be reproduced over $M_g^{\trop}$ only in the nondegenerate case.  This paper is dedicated to the construction of a universal tropical Jacobian over $M_g^{\trop}$ as a generalized cone complex, and it is a natural completion of the results in \cite{AP1}. 

\subsection{Outline of the results}

The objects parametrized in Caporaso space $\ol{P}_{d,g}$  are  stably balanced line bundle on quasistable curves.  Later, Pandharipande re-interpreted $\ol{P}_{d,g}$ as the space parametrizing polystable torsion-free rank-1 sheaves on stable curves. We introduce their tropical analogues, that are  $\mu$-polystable  (pseudo-)divisors on graphs and tropical curves. Here, $\mu$ denotes a polarization. 
 
 The key ingredient to construct a moduli space via polystable divisors is that 
  every divisor on a tropical curve is equivalent to a $\mu$-polystable divisor. Moreorever, two equivalent $\mu$-polystable divisors have the same combinatorial type (see Theorem \ref{thm:poly}). Our proof relies on  purely combinatorial arguments. The same result is also proved in \cite[Proposition 4.4]{CPS} using a more geometric approach.  
%that every divisor on a tropical curve has a polystable representative, and two equivalent polystable divisors have the same combinatorial type.
We use this result to construct a generalized cone  complex $P^\trop_{\mu,g}$ over $M_g^{\trop}$. Each cone parametrizes equivalence classes of polystable divisors on tropical curves with fixed combinatorial type. We need  a subtle analysis of how the cones can be glued in a nice way as prescribed by specializations of graphs (see Example \ref{exa:Karl}).

\begin{Thm*}[\ref{prop:parameteruni}  and \ref{thm:universalPic}]
The generalized cone complex $P^\trop_{\mu,g}$ has dimension $4g-3$ and it is connected in codimension $1$. The natural forgetful map $\pi^{\trop}\col P^\trop_{\mu,g}\to M_g^\trop$ is a map of generalized cone complexes, and we have: 
\[
(\pi^{\trop})^{-1}[X]\cong J(X)/\Aut(X),
\]
for every stable tropical curve $X$ of genus $g$. (Here, $J(X)$ is the tropical Jacobian.)
The space $P^\trop_{\mu,g}$ parametrizes equivalence classes of pairs $(X,\D)$, where $X$ is a stable tropical curve of genus $g$ and $\D$ is a $\mu$-polystable divisor on $X$.
\end{Thm*}

We point out that,
looking back to algebraic geometry, there are other natural candidates for a universal tropical Jacobian over $M_g^{\trop}$ (see Section \ref{sec:attempt}). None of them are suitable for our purposes. In fact, one could try to define a universal object parametrizing \emph{all} $\mu$-semistable divisors on tropical curves. The problem is that the corresponding topological space is not a cone complex. One could try to solve this issue by considering just \emph{simple} $\mu$-semistable divisors. In this case, even though the corresponding topological space is a generalized cone complex, its points are not in one to one correspondence with linear equivalence classes of divisors on tropical curves.

We also construct polyhedral decompositions of the Jacobian $J(X)$ of a tropical curve $X$. 
Polyhedral decompositions of the tropical Jacobian are  closely related to toroidal compactifications of the Jacobian of a curve (see \cite{Mum}, \cite{MW}, \cite{CPS}). An interesting decomposition is studied in \cite{ABKS} in degree $g$ through break divisors. A similar analysis is done in \cite{AP1} for all degrees through quasistable divisors on tropical curves. The strategy is to define a polyhedral complex  $P^{\trop}_\mu(X)$ by means of  $\mu$-polystable divisors on $X$. This is done by gluing polytopes along faces as prescribed by specializations. Then we compare $P_\mu^{\trop}(X)$ with $J(X)$:

\begin{Thm*}[\ref{thm:decomposition}]
Given a tropical curve $X$, there is a homeomorphism $P^{\trop}_\mu(X)\to J(X)$. 
\end{Thm*}

It is interesting to observe that, if $p_0$ is a point of $X$ and $J^{\trop}_{p_0,\mu}(X)$ is the polyhedral complex of $(p_0,\mu)$-quasistable divisors on $X$ (also homeomorphic to $J(X)$), then we have a refinement map of polyhedral complexes $J^{\trop}_{p_0,\mu}(X)\to P^{\trop}_\mu(X)$ (see Proposition \ref{prop:refinement}). The same polyhedral decompositions of the tropical Jacobian by means of polystable divisors and the study of their relationship with quasistable divisors also appear in \cite[Proposition 5.8 and Corollary 5.10]{CPS}.

A tool to compare quasistable (or, more generally, semistable) divisors and polystable divisors is Proposition \ref{prop:minimal}. There, we show that for every $\mu$-semistable pseudo-divisor on a graph $\Gamma$, there is a unique minimal $\mu$-polystable pseudo-divisor on $\Gamma$ that specializes to it. This distinguished $\mu$-polystable pseudo-divisor is obtained  via an iterative procedure reminding of the construction of the graded of a torsion-free rank-1 sheaf on a curve by means of Jordan-Holder filtrations (see \cite[Section 1.3]{Es01} and \cite[Proposition 3.7]{CPS}).  Proposition \ref{prop:minimal} can be viewed as a discrete analogue of the property of a geometric invariant theory quotient stating that 
each semistable orbit contains a unique polystable orbit in its closure.

In Section \ref{sec:strat} we discuss natural stratifications (in the sense of \cite[Definition 1.4.2]{CC}) of universal Picard moduli spaces. 

%We conclude the paper with a question about the relationship of $P_{\mu,g}^{\trop}$ and skeletons of universal Picard stacks. 

About ten days before submitting this paper to ArXiv, we were made aware of the paper \cite{CPS} by Christ, Payne, and Shen. They also consider polystable divisors on tropical curves to describe polyhedral decompositions of the tropical Jacobian, and relate them to Mumford models of compactified Jacobians. Some of the present results appear in \cite{CPS}, although the methods and purpose of the two papers are somewhat different. We made an effort to highlight the common results.

\subsection*{Acknowledgments}	

We thank Karl Christ for sending us the preprint \cite{CPS}. We  thank Karl Christ and Sam Payne for important comments and remarks on a preliminary version of this paper. 

The second and fourth author were supported by Capes (Bolsa de doutorado), the third author was supported by CNPq-APQ, processo 301314/2016-0.

\section{Preliminaries}

\subsection{Posets}
 A \emph{poset} (\emph{partially ordered set}) is a pair $(S,\leq)$ where $S$ is a set and $\leq$ is a partial order on $S$.
  A \emph{chain} in $S$ is a sequence $x_0<x_1<\ldots < x_n$. We call $n$ the \emph{length} of the chain.
%  and we say that $x_0$ (respectively, $x_n$) is the \emph{starting point} (respectively, \emph{ending point}) of the chain. 
We say that $S$ is \emph{ranked} (of length $n$) if the maximal chains have all the same length $n$. The  maximal length of chains in $S$ is precisely the Krull dimension of $S$ as a topological space. Hence, if $S$ is ranked, then it is of pure dimension.
%For every $x\in S$ we define the \emph{dimension} of $x$ as the maximum length for a chain ending in $x$; in other words, the dimension of $x$ is precisely the Krull dimension of $\overline{\{x\}}$.
If $S$ is a ranked poset of length $n$, then the \emph{codimension} of $x\in S$ is $n-\dim_S(x)$, i.e., the length of all maximal chains starting from $x$. 
We say that $S$ is \emph{graded} if it has a function $\text{rk}\col S\to \mathbb{N}$, called \emph{rank function}, such that $\text{rk}(x)=\text{rk}(y)+1$ whenever $y<x$ and there is no element $z\in S$ such that $y<z<x$. Every ranked poset is graded with rank function given by the dimension.

Let $S$ be a ranked poset. We say that $S$ is \emph{connected in codimension one} if for every maximal elements $y,y'\in S$ there are two sequences of elements $x_1,\ldots, x_n\in S$ and $y_0,\ldots, y_n\in S$ such that 
	\begin{enumerate}
	\item $x_i$ has codimension $1$ for every $i=1,\ldots, n$.
	\item $y_i$ is maximal for every $i=0,\ldots,n$.
	\item $y_0=y$ and $y_n=y'$.
	\item $x_{i+1}< y_i$ for every $i=0,\ldots,n-1$ and $x_i<y_i$ for every $i=1,\ldots, n$.
	\end{enumerate}
We call these sequences a \emph{path in codimension $1$} from $y$ to $y'$.	

\subsection{Cones and polyhedra}

We briefly introduce cones and polyhedra.
 We will adopt the terminology of \cite[Section 3.2]{AP1}. 
 Given a finite set $S\subset \R^n$ we define 
\[
\cone(S):=\left\{\sum_{s\in S}\lambda_ss|\lambda_s\in \mathbb R_{\ge0}\right\}.
\]
A subset $\sigma\subset \R^n$ is called a \emph{polyhedral cone} if $\sigma=\cone(S)$ for some finite set $S\subset \mathbb{R}^n$. If there is a subset $S\subset \mathbb{Z}^n$ with $\sigma=\cone(S)$ then $\sigma$ is called \emph{rational}. Throughout,  \emph{cone} will mean \emph{rational polyhedral cone}. A \emph{face} of a cone $\sigma$ is the intersection of $\sigma$ with some linear subspace $H\subset \mathbb{R}^n$ of codimension one such that $\sigma$ is contained in one of the closed half-spaces determined by $H$. A \emph{generalized cone complex} is the colimit (as topological space) of a finite diagram of cones with face morphisms (i.e., morphisms of cones taking faces to faces). We refer to \cite[Section 2]{ACP} for the more details on cone complexes.

A \emph{polyhedron} $P\subset \R^n$ is an intersection of a finite number of half-spaces of $\R^n$. A face of a polyhedron $P$ is the intersection of $P$ and a hyperplane $H$ such that $P$ is contained in a closed half-space determined by $H$.  A \emph{polyhedral complex} is the colimit (as topological space) of a finite poset of polyhedra with face morphisms (i.e., morphisms of polyhedra taking faces to faces).

\subsection{Graphs}

  Let $\Gamma$ be a graph. We denote by $V(\Gamma)$ the set of vertices and by $E(\Gamma)$ the set of edges of $\Gamma$. We also denote by $b_0(\Gamma)$ and $b_1(\Gamma)$ its first and second Betti numbers, i.e., $b_0(\Gamma)$ is the number of connected components and 
  \[
  b_1(\Gamma):=|E(\Gamma)|-|V(\Gamma)|+b_0(\Gamma).
  \]
Sometimes, we will refer to $b_1(\Gamma)$ as the \emph{genus} of the graph. We let $\Aut(\Gamma)$ be the group of automorphisms of $\Gamma$.

Given a subset $V\subset V(\Gamma)$, we let $E(V)\subset E(\Gamma)$ be the subset of edges of $\Gamma$ connecting vertices in $V$.    
  Given disjoint subsets $V,W\subset V(\Gamma)$ of $V(\Gamma)$ we define the set $E(V,W)\subset E(\Gamma)$ of the edges that connect a vertex in $V$ to a vertex in $W$. We set $\delta_{\Gamma,V}:=|E(V,V^c)|$. More generally, given subsets $V,W\subset V(\Gamma)$, we define $E(V,W):=E(V\setminus W,W\setminus V)$.
  
  Given a subset $\E\subset E(\Gamma)$ and a vertex $v$ of $\Gamma$, we define the \emph{valence of $v$ in $\E$}, denoted by $\val_\E(v)$, as the number of edges in $\E$ incident to $v$ (with loops counting twice). In the case that $\E=E(\Gamma)$, we simply write $\val(v)$ and call it the \emph{valence} of $v$.
For every subset $V\subset V(\Gamma)$, we set  $\val_\E(V):=\sum_{v\in V}\val_\E(v)$.

The graph $\Gamma$ is \emph{circular} if it is connected and its vertices have all valence $2$. A \emph{cycle} on $\Gamma$ is a circular subgraph of $\Gamma$. 
  
  Consider a subset $\E\subset E(\Gamma)$. 
  We let $\Gamma/\E$ and $\Gamma_\E$ be the graphs obtained by the contraction of edges in $\E$ and by the removal of edges in $\E$, respectively. There is a natural surjection $V(\Gamma)\to V(\Gamma/\E)$ and a natural identification $E(\Gamma/\E)=E(\Gamma)\setminus\E$. Moreover, we have $V(\Gamma_\E)=V(\Gamma)$ and $E(\Gamma_\E)=E(\Gamma)\setminus\E$. 
  The subset $\E$ is called \emph{non-disconnecting} if $\Gamma_\E$ is connected, otherwise it is \emph{disconnecting}.

  We let $\Gamma^\E$ be the graph obtained from $\Gamma$ by inserting a vertex, called \emph{exceptional}, in the interior of any edge $e\in\E$. We denote by $v_e$ the new vertex inside $e$.  Thus, for every edge $e\in \E$, we get exactly two edges $e_1,e_2$ of $\Gamma^\E$ incident to $v_e$. We say that $e_1,e_2$ are the edges \emph{over} $e$, and that $e$ is the edge \emph{under} $e_1$ (or $e_2$). A \emph{refinement} of $\Gamma$ is a graph obtained by iterating the operation taking $\Gamma$ to $\Gamma^\E$.

   A graph $\Gamma$ \emph{specializes} to a graph $\Gamma'$, and we write $\iota\col \Gamma\ra\Gamma'$, if there is a subset $\E\subset E(\Gamma)$ such that $\Gamma'$ is isomorphic to $\Gamma/\E$. Then a specialization $\iota\col \Gamma\ra\Gamma'$ comes equipped with a surjective map $\iota^V\col V(\Gamma)\ra V(\Gamma')$ and an injective map $\iota^E\col E(\Gamma')\ra E(\Gamma)$. We simply write $\iota=\iota^V$ and see $E(\Gamma')$ as a subset of $E(\Gamma)$ via $\iota^E$.

  Sometimes we consider an orientation on $\Gamma$. In this case, for every edge $e\in E(\Gamma)$, we denote by $s(e),t(e)\in V(\Gamma)$ the source and the target of an (oriented) edge $e\in E(\Gamma)$, respectively. For every subset $\E\subset E(\Gamma)$ and every edge $e\in E(\Gamma)$, we  let $e^s,e^t$ the edges of $\Gamma^\E$ over $e$, with $e^s$ incident to $s(e)$ and $e^t$ incident to $t(e)$, respectively.
    Moreover, given an oriented cycle $\gamma$  on $\Gamma$, we define 
\begin{equation}\label{eq:gammadef}
\gamma(e)=\begin{cases}
 0 & \text{ if $e$ is not a edge of $\gamma$}\\
 1 & \text{ if the orientations on $\gamma$ and $\Gamma$ coincide on $e$}\\
 -1& \text{otherwise.}
\end{cases}
\end{equation}

A \emph{(vertex) weighted graph} is a graph $\Gamma$ together with a function $w_\Gamma\colon V(\Gamma)\to\mathbb{Z}_{\geq0}$, called \emph{weight function}. A weighted graph $\Gamma$ is \emph{stable} if $\val(v)+2w(v)\geq3$ for every vertex $v\in V(\Gamma)$.  The genus of $\Gamma$ is $g(\Gamma):=\sum_{v\in V(\Gamma)} w_\Gamma(v)+b_1(\Gamma)$.  
A \emph{specialization} $\iota\col \Gamma\ra \Gamma'$ of weighted graphs $\Gamma$ and $\Gamma'$ is a graph specialization such that $w_{\Gamma'}(v')=w_\Gamma(\iota^{-1}(v'))$ for every $v'\in V(\Gamma')$.

\subsection{Tropical curves}

A \emph{tropical curve} is a metric space $X$ such that there exists a weighted graph $\Gamma$ and a function $\ell\col \Gamma\ra \R^{E(\Gamma)}$, called \emph{length function}, so that $X$ is obtained by gluing segments $[0,\ell(e)]\subset \R$ for every $e\in E(\Gamma)$ at their end vertices, as prescribed by the combinatorial data of the graph. We call $(\Gamma,\ell)$ (or simply $\Gamma$ when the function $\ell$ is clear), \emph{a model} of the tropical curve $X$. We will identify isometric tropical curves, so a tropical curve could admit different models. For every edge $e\in E(\Gamma)$, we let $e^\circ\subset X$ be the interior of the corresponding segment of $X$. For points $p,q\in e$, we denote by $\ol{pq}$ the segment contained in $e$ with endpoints $p$ and $q$. If an orientation is chosen on $\Gamma$, we denote by $\overrightarrow{pq}$ the oriented segment.

Let $X$ be a tropical curve. 
 The \emph{genus} of $X$ is the genus of one of its underlying weighted graphs. The \emph{valence} of a point $p\in X$ is the valence of $p$ as a vertex of any model containing $p$ as a vertex. Note that $X$ inherits a well-defined weight function $w_X\col X\ra \mathbb{Z}_{\ge0}$ from any one of its models, where $w_X(p)=0$ is $p$ is not a vertex of the model. We say that $X$ is \emph{stable} if $\delta_{X,p}+2w_X(p)\ge 3$ for every point $p\in X$ such that $\delta_{X,p}\leq1$.  The \emph{stable model} of $X$ is the model $\Gamma$ such that
\[
V(\Gamma)=\{p\in X ; \text{ either } \val_X(p)\ne 2, \text{ or } \val_X(p)=2 \text{ and } w_X(p)\ne 0\}.
\]

 A \emph{polarization of degree $d$} on $X$ is a function $\mu\col X\to\R$ such that $\mu(p)=0$ for all, but finitely many $p\in X$, and $\sum_{p\in X}\mu(p)=d$. Note that when $X$ has model $\Gamma$ and $\mu$ is a polarization on $\Gamma$, then there is a natural induced polarization on $X$. Conversely, all polarizations on $X$ come from polarizations on models of $X$.   
 %Throughout the paper, when a tropical curve has a polarization $\mu$ and model $\Gamma$, we will implicitly assume that $\mu$ is induced by a polarization on $\Gamma$.
 
 Given a polarization $\mu$ on $X$, the \emph{$\mu$-model} of $X$ is the model whose vertices are the vertices of its stable model and the points $p$ of $X$ such that $\mu(p)\ne 0$.

A \emph{(tropical) subcurve} $X$ is a tropical curve $Z$ admitting an injection $Z\subset X$ that is an isometry over each connected component of $Z$. If $\Gamma$ is a model of $X$ and $Z$ is a subcurve of $X$, then there exists a minimal refinement $\Gamma'$ of $\Gamma$ such that $Z$ is induced by a subgraph $\Gamma'_Z$ of $\Gamma'$. We define
\[
\delta_{X,Z}:=\sum_{v\in V(\Gamma'_Z)}\val_{E(\Gamma')\setminus E(\Gamma'_Z)}(v)
\]

Given a subset $V\subset V(\Gamma)$, for a model $\Gamma$ of $X$, we define $X_V$ as the subcurve of $X$ with model $\Gamma(V)$ and length function induced by the lenght function of $X$.

\section{Divisors on graphs and tropical curves}  
  
  \subsection{Divisors on graphs}

Let $\Gamma$ be a (weighted) graph. A divisor $D$ on $\Gamma$ is a function $D\colon V(\Gamma)\to \mathbb{Z}$. The degree of $D$ is the integer $\deg D:=\sum_{v\in V(\Gamma)}D(v)$. 

Let $D$ be a divisor on $\Gamma$. Given  a subset $\E\subset E(\Gamma)$, we define the divisor $D_\E$ on $\Gamma_\E$ as $D_\E(v):=D(v)$, for every $v\in V(\Gamma_\E)=V(\Gamma)$.
If $\iota\col\Gamma\ra \Gamma'$ is a specialization of graphs, we define the divisor $\iota_*(D)$ on $\Gamma'$ such that $\iota_*(D)(v')=\sum_{v\in\iota^{-1}(v')}D(v)$.

	A \emph{pseudo-divisor} on $\Gamma$ is a pair $(\E,D)$ where $\E\subset E(\Gamma)$ and $D$ is a divisor on $\Gamma^\E$ such that $D(v)=-1$ for every exceptional vertex $v\in V(\Gamma^\E)$. If $\E=\emptyset$, then $(\E,D)$ is just a divisor of $\Gamma$.
	We denote by $\Aut(\Gamma,\E,D)$ the subgroup of $\Aut(\Gamma)$ made of maps sending $(\E,D)$ to itself. More generally, two triples $(\Gamma,\E,D)$ and $(\Gamma',\E',D')$ are \emph{isomorphic} if there is a graph   isomorphism from $\Gamma$ to $\Gamma'$ mapping $(\E,D)$ to $(\E',D')$.
	
	If $\iota\col\Gamma\ra \Gamma'$ is a specialization of graphs and $(\E,D)$ is a pseudo-divisor on $\Gamma$, we define $\iota_*(\E,D)$ as the pseudo-divisor on $\Gamma'$ given by $(\E\cap E(\Gamma'),(\iota_{\E})_*(D))$, where $\iota_\E\col \Gamma^\E\to \Gamma'^{\E\cap E(\Gamma')}$ is the specialization induced by $\iota$.

Given an integer $d$, a \emph{degree-$d$ polarization} on $\Gamma$ is a function $\mu\col V(\Gamma)\to\R$ such that $\sum_{v\in V(\Gamma)}\mu(v)=d$. If $\mu$ is a polarization on $\Gamma$, then we set $\mu(V):=\sum_{v\in V}\mu(v)$ for every subset $V\subset V(\Gamma)$. For every specialization of graphs $\iota\col\Gamma\to\Gamma'$, there is an induced degree-$d$ polarization $\iota_*(\mu)$ on $\Gamma'$ defined as $\iota_*(\mu)(v'):=\sum_{v\in\iota^{-1}(v')}\mu(v)$.  

A \emph{universal genus-$g$ polarization (of degree $d$)} is the datum of a polarization $\mu_\Gamma$ for every genus-$g$ stable weighted graph $\Gamma$ such that $\mu_{\Gamma'}=\iota_*(\mu_\Gamma)$ for every specialization $\Gamma\ra \Gamma'$. 
The \emph{canonical genus-$g$ universal polarization of degree $d$} is
\[
\mu_\Gamma(v)=\frac{d(2w_\Gamma(v)-2+\val(v))}{2g-2}.
\]

Let $\mu$ be a polarization of degree $d$ on $\Gamma$.
If $\E\subset E(\Gamma)$ is a subset of edges, then $\mu_\E(v):=\mu(v)+\frac{1}{2}\val_\E(v)$ defines a polarization $\mu_\E$ on $\Gamma_\E$ of degree $d+|\E|$.
Given a subdivision $\Gamma^{\E}$ of $\Gamma$ for some $\E\subset E(\Gamma)$, there is an induced degree-$d$ polarization $\mu^\E$ on $\Gamma^\E$ such that $\mu^\E(v)=\mu(v)$ if $v\in V(\Gamma)$, and $\mu^\E(v)=0$ otherwise.

Let $\Gamma$ be a connected graph. Let $\mu$ be a degree-$d$ polarization on $\Gamma$ and $D$  a degree-$d$ divisor on $\Gamma$. For every subset $V\subset V(\Gamma)$, we set
\[
\beta_D(V)=\deg(D|_V)-\mu(V)+\frac{\delta_{\Gamma,V}}{2}.
\]
By \cite[Lemma 4.1]{AP1}, for subsets $V,W\subset V(\Gamma)$ we have:
\begin{equation}
\label{eq:betasumD}
\beta_D(V\cup W)+\beta_D(V\cap W)=\beta_D(V)+\beta_D(W)-|E(V, W)|.
\end{equation}

  The divisor $D$ is $\mu$-\emph{semistable} (respectively, $\mu$-\emph{stable}) on $\Gamma$ if $\beta_D(V)\geq0$ (respectively, $\beta_D(V)>0$)  for every subset $V\subset V(\Gamma)$ (respectively, for every non-empty proper subset  $V\subsetneqq V(\Gamma)$). 
  Given a vertex $v_0\in V(\Gamma)$, the divisor $D$ is $(v_0,\mu)$-\emph{quasistable} if $\beta_D(V)\geq0$ for every proper subset $V\subsetneqq V(\Gamma)$, with strict inequality if $v_0\in V$.\par
  
   A polarization $\mu$ is said  to be \emph{nondegenerate} if every $\mu$-semistable divisor  is actually $\mu$-stable (see \cite{KP} for a more explicit characterization). A universal polarization is said to be \emph{nondegenerate} if it is nondegenerate for each graph.
  
  We need to extend the stability conditions to the case of non-connected graphs. 
  The notion of $\mu$-semistability naturally extends for divisors on non-connected graphs.

\begin{Rem}
  Let $\Gamma$ be a graph. Let $\mu$ be a degree-$d$ polarization on $\Gamma$ and $D$ a $\mu$-semistable divisor of degree-$d$ on $\Gamma$. If $\Gamma'$ is a connected component of $\Gamma$, then 
  \[
  \beta_D(V(\Gamma'))+\beta_D(V(\Gamma')^c)=0.
  \]
  Hence $\beta_D(V(\Gamma'))=0$. In particular,  $\mu(\Gamma')=\deg(D|_{\Gamma'})$, which is an integer, so  $\mu|_{\Gamma'}$ is a polarization on $\Gamma'$.
\end{Rem}
 
 If $\Gamma$ is a non-connected graph and $\mu$ is  a degree-$d$ polarization on $\Gamma$, we say that a divisor $D$ of degree $d$ on $\Gamma$ is $\mu$-\emph{stable} if it is $\mu$-semistable and $D|_{\Gamma'}$ is $\mu|_{\Gamma'}$-stable for every connected component $\Gamma'$ of $\Gamma$. 
 
  More generally, let $(\E,D)$ be a pseudo-divisor on $\Gamma$ and let $\mu_\E$ be the induced polarization on $\Gamma_\E$. For every subset $V\subset V(\Gamma_\E)=V(\Gamma)$, we set
\[
\beta_{\E,D}(V)=\deg(D|_V)-\mu_\E(V)+\frac{\delta_{\Gamma_\E,V}}{2}.
\]
Notice that $\beta_{\E,D}(V)=\beta_D(V)$ if $\E=\emptyset$. \par
   For subsets $\E,\E'\subset E(\Gamma)$ such that $\E'\subset \E$, we can consider the pseudo-divisor $(\E',D')$ on $\Gamma_{\E\setminus\E'}$, where $D'=D_{\E\setminus\E'}$. We have $D_\E =(D_{\E\setminus \E'})_{\E'}$, and hence
\begin{equation}
\label{eq:betaee}
\beta_{\E',D_{\E\setminus\E'}}(V)=\beta_{\E,D}(V).    
\end{equation}

Taking $\E'=\emptyset$ in Equation \eqref{eq:betaee},
we get an analogue of Equation \eqref{eq:betasumD} for pseudo-divisors:  for every subsets $V,W\subset V(\Gamma_\E)=V(\Gamma)$, 
\begin{equation}
\label{eq:betasum}
\beta_{\E,D}(V\cup W)+\beta_{\E,D}(V\cap W)=\beta_{\E,D}(V)+\beta_{\E,D}(W)-|E(V, W)\setminus \E|.
\end{equation}

We could have defined $\beta_{\E,D}(V)$ in terms of the refined graph $\Gamma^\E$. Indeed, for every subset $V\subset V(\Gamma_\E)$, we have that $\beta_{\E,D}(V)=\beta_{D}(\widetilde{V})$,
where $\widetilde{V}\subset V(\Gamma^\E)$ is: 
\[
\widetilde{V}:=V \cup\{v_e;e\in \E\cap(E(V)\cup E(V,V^c))\}.
\]
Note that for every subset $V\subset V(\Gamma^\E)$, 
\begin{equation}
    \label{eq:Wtil}
    \beta_D(V)\geq \beta_D(\widetilde{V\cap V(\Gamma)})
\end{equation}
with equality if and only if $V=\widetilde{V\cap V(\Gamma)}$.

  We say that a pseudo-divisor $(\E,D)$ is \emph{simple} if $\E$ is non-disconnecting.  We say that $(\E,D)$ is $\mu$-\emph{semistable} if $\beta_{\E,D}(V)\geq 0$  for every subset $V\subset V(\Gamma)$. One can see that $(\E,D)$ is semistable if and only  if $D_\E$ is $\mu_\E$-semistable on $\Gamma_\E$, or, equivalently, if and only if $D$ is $\mu^\E$-semistable on $\Gamma^\E$ (see \cite[Proposition 4.6]{AP1}).

  Given a graph $\Gamma$ we define the poset $\pd(\Gamma)$ of pseudo-divisors on $\Gamma$ with partial order $(\E,D)\geq(\E',D')$ if $(\E,D)$ specializes to $(\E',D')$.
  We also define the category $\PD_g$ whose objects are triples $(\Gamma,\E,D)$, where $\Gamma$ is a stable weighted graph of genus $g$  and $(\E,D)$ is a pseudo-divisor on $\Gamma$, and the morphisms are given by specializations. We let $\pd_g$ be the poset $\PD_g/\sim$ where $(\Gamma,\E,D)\sim(\Gamma',\E',D')$ if they are isomorphic.
  \begin{Rem}
  \label{rem:semispec}
  An easy adaptation of \cite[Proposition 4.6]{AP1} implies that semistability is preserved under graph specialization.
  \end{Rem}
  
  Given a graph $\Gamma$, a degree-$d$ polarization $\mu$ and a vertex $v_0\in V(\Gamma)$, we define $\sst_\mu(\Gamma)$ and 
  %$\ps_\mu(\Gamma)$ and 
  $\qs_{v_0,\mu}(\Gamma)$ as the subposets of $\pd(\Gamma)$  of $\mu$-semistable %$\mu$-polystable 
  and $(v_0,\mu)$-quasistable pseudo-divisors, respectively. 
  
%  If $\mu$ is a genus-$g$ universal polarization, we also define the categories $\SST_{\mu,g}$ (respectively,  $\QS_{\mu,g}$) whose objects are triples $(\Gamma,\E,D)$ where $\Gamma$ is a genus-$g$ stable weighted graph (respectively, a genus-$g$ stable weighted graph with 1 leg $v_0$) and $(\E,D)$ is a $\mu$-semistable (respectively, $(v_0,\mu)$-quasistable) pseudo-divisor on $\Gamma$, and morphisms are specializations. We let $\sst_{\mu,g}$ and $\qs_{\mu,g}$ be the posets $\SST_{\mu,g}/\sim$ and $\QS_{\mu,g}/\sim$, where the relations $\sim$ are given by isomorphisms.
 % \footnote{I think we can erase this paragraph}
  %By Remark \ref{rem:polysemi} we have natural inclusions $\qs(\Gamma)\subset \sst(\Gamma)$ and $\ps(\Gamma)\subset \sst(\Gamma)$. We say that a polarization $\mu$ on $\Gamma$ is \emph{nondegenerate} if the inclusions above are equalities \cite{KP}.

  We define the \emph{rank} map $\rk\col \pd(\Gamma)\to \mathbb{N}$ as 
  \[
  \rk(\E,D)=|\E|-b_0(\Gamma_\E)+b_0(\Gamma).
  \]
  Note that $\rk(\E,D)\leq b_1(\Gamma)$. Also, if $(\E,D)\geq(\E',D')$ then $\rk(\E,D)\geq \rk(\E',D')$.
  
  \begin{Lem}
  \label{lem:middle}
  Let $\Gamma$ be a graph.
  Let $(\E_1,D_1)$, $(\E_2,D_2)$  be pseudo-divisors on $\Gamma$ such that $(\E_1,D_1)\geq (\E_2,D_2)$. For each subset $\E\subset E(\Gamma)$ such that $\E_2\subset \E\subset \E_1$, there exists a unique pseudo-divisor $(\E,D)$ such that 
  \[
  (\E_1,D_1)\geq (\E,D)\geq (\E_2,D_2).
  \]
  In particular if $(\E_1,D_1)$ is $\mu$-semistable for some polarization $\mu$ on $\Gamma$, then $(\E,D)$ is $\mu$-semistable.
  \end{Lem}
  \begin{proof}
    The condition $(\E_1,D_1)\geq (\E_2,D_2)$ gives rise to a graph specialization $\iota\col\Gamma^{\E_1}\to \Gamma^{\E_2}$. Note that $\iota$ factors through unique specializations $\iota_1\col\Gamma^{\E_1}\to \Gamma^{\E}$ and $\Gamma^{\E}\to \Gamma^{\E_2}$. It is enough to take the divisor $D=\iota_{1*}(D_1)$ on $\Gamma^{\E}$. The last sentence follows from 
     Remark \ref{rem:semispec}.
  \end{proof}

  %Let $\iota\col (\E_1,D_1)\to(\E_2,D_2)$ be a specialization of pseudo-divisors on a graph $\Gamma$. For each subset $V\subset V(\Gamma)$, we define 
 %\begin{equation}
%\label{eq:}
% E(V,\iota)=\{e \in E(V,V^c)\cap \E_1; \iota(v_e)\in V\}.
% \end{equation}
% Note that $E(V,\iota)\cap \E_2=\emptyset$.

\begin{Lem}
\label{lem:betaspec1}
Let $\Gamma$ be a graph and $\mu$ be a degree-$d$ polarization on $\Gamma$. Let $\iota\col (\E_1,D_1)\to (\E_2,D_2)$ be a specialization of degree-$d$ pseudo-divisors on $\Gamma$. For each $V\subset V(\Gamma)$, we have $\beta_{\E_1,D_1}(V)\leq \beta_{\E_2,D_2}(V)$.
\end{Lem}
\begin{proof}
%\[
%\beta_{\E_1,D_1}(V)= %\beta_{\E_2,D_2}(V)+|E(V,\iota)|-|E(V,V^c)\cap (\E_1\setminus \E_2)|.
%\]
%In particular, 
We compute:
\begin{align*}
\beta_{\E_1,D_1}(V)&=\deg(D_1|_V)-\mu_{\E_1}(V)+\frac{\delta_{\Gamma_{\E_1},V}}{2}\\
                   &=\deg(D_1|_V)-\left(\mu_{\E_2}(V)+\frac{\val_{\E_1\setminus\E_2}(V)}{2}\right)+\\
                   &\;\;\;\;\;+\frac{\delta_{\Gamma_{\E_2},V}-\val_{\E_1\setminus\E_2}(V)}{2}\\
                   &=\beta_{\E_2,D_2}(V)+\deg(D_1|_V)-\deg(D_2|_V)-\val_{\E_1\setminus\E_2}(V).
\end{align*}
However, $D_2(v)\geq D_1(v)-\val_{\E_1\setminus\E_2}(v)$ for every $v\in V$. The result follows.
\end{proof}

\begin{Lem}
\label{lem:betaspec2}
 Let $\iota\col \Gamma_1\to \Gamma_2$ be a specialization of graphs and let $(\E_1,D_1)$ be a pseudo-divisor on $\Gamma_1$. Define $(\E_2,D_2)=\iota_*(\E_1,D_1)$. For every $V\subset V(\Gamma_2)$, we have
 \[
 \beta_{\E_2,D_2}(V)=\beta_{\E_1,D_1}(\iota^{-1}(V)).
 \]
\end{Lem}
\begin{proof}
 Note that $\E_2=\E_1\cap E(\Gamma_2)$ and 
 \begin{align*}
  \deg(D_2|_V) &= \deg(D_1|_{\iota^{-1}(V)})-|(\E_1\setminus E(\Gamma_2))\cap E(\iota^{-1}(V))|,  \\
 \iota_*(\mu)_{\E_2}(V)& = \mu_{\E_1}(\iota^{-1}(V))-|(\E_1\setminus E(\Gamma_2))\cap E(\iota^{-1}(V))|,\\ \delta_{\Gamma_2,V}& = \delta_{\Gamma_1,\iota^{-1}(V)}.
 \end{align*}
 (In the formula we consider $E(\Gamma_2)$ as a subset of $E(\Gamma_1)$.)
The result follows.
\end{proof}

\subsection{Divisors on a tropical curve} 
 
 Let $X$ be a tropical curve.  A \emph{divisor} on $X$ is a map $\D\col X\to \mathbb{Z}$ such that $\D(p)\neq0$ for finitely many points $p\in X$. The degree of a divisor $\D$ on $X$ is the integer $\deg \D:=\sum_{p\in X} \D(p)$. We say that $\D$ is \emph{effective} if $\D(p)\ge0$ for every $p\in X$. The set $\Div(X)$ of divisors on $X$ is an Abelian group.
 
 A \emph{rational function} on $X$ is a continuous, piece-wise linear function $f\col X\ra \R$ with integer slopes. A \emph{principal divisor} on $X$ is a divisor of type
\[
\Div_X(f):=\sum_{p\in X} \ord_p(f) p\in \Div(X),
\]
where $f$ is a rational function on $X$ and $\ord_p(f)$ is the sum of the incoming slopes of $f$ at $p$. A principal divisor has degree zero. The set $\Prin(X)$ of principal divisors on $X$ is a subgroup of $\Div(X)$.

Given divisors $\D_1,\D_2$ on $X$, we say that $\D_1$ and $\D_2$ are \emph{equivalent} if $\D_1-\D_2$ is a principal divisor. The \emph{Picard group} $\Pic(X)$ of $X$ is defined as
\[
\Pic(X):=\Div(X)/\Prin(X).
\] 
The \emph{degree-$d$ Picard group} of $X$ is $\Pic^d(X):=\Div^d(X)/\Prin(X)$.
 
 The \emph{Jacobian} of $X$ is the real torus defined as:
 \[
 J(X)=\Omega(X)^\vee/H_1(X,\mathbb Z)
 \]
 where $\Omega(X)$ is the space of harmonic $1$-forms on $X$ (we refer to \cite[Section 3]{BF} and \cite[Section 6]{MZ} for more details). Recall that there is a canonical isomorphism between $\Pic^0(X)$ and $J(X)$.
 
 \subsection{Unitary divisors on tropical curves}
 
 In this paper we will restrict our attention to a special type of divisors on a tropical curve, called unitary divisors. These divisors will be enough later to define a universal tropical Jacobian.

  \begin{Def}
  Let $X$ be a tropical curve and 
   $\Gamma$ a model of $X$. A divisor $\D$ on $X$ is a \emph{$\Gamma$-unitary divisor} (or simply \emph{unitary divisor} if the model of $X$ is clear) if for every $e\in E(\Gamma)$ we have $\D(p)=0$ for each point $p\in e^\circ$, except for at most one point $p_{\D,e}\in e^0$ for which  $\D(p_{\D,e})=-1$.  
  \end{Def}

  Throughout the paper, we will use the notation $p_{\D,e}$ to denote the point of $e^0$ (if it exists) such that $\D(p_{\D,e})=-1$. 
  
  Let us define the combinatorial type of  a unitary divisor $\D$ on $X$. First, we let 
  \[
  \E=\{e\in E(\Gamma); \exists\; p\in e^\circ \text{ such that }\D(p)=-1\}.
  \]
 Then, we define the divisor $D$ on $\Gamma^\E$ such that $D(v)=\D(v)$ for every $v\in V(\Gamma)$ and $D(v_e)=-1$ for every exceptional vertex $v_e\in V(\Gamma^\E)$. In this way, we obtain a pseudo-divisor $(\E,D)$ on $\Gamma$, called
  the \emph{combinatorial type} of $\D$.

Let $X$ be a tropical curve and $\mu$ be a degree $d$ polarization on $X$. From now on, we  let $\Gamma$ be the $\mu$-model of $X$.
  For a divisor $\D$ on $X$ and a subcurve $Z\subset X$, we let
  \[
  \beta_\D(Z)=\deg(\D|_Z)-\mu(Z)+\frac{\delta_{X,Z}}{2}.
  \]
  We say that $\D$ is $\mu$-semistable if $\beta_\D(Z)\geq0$ for every $Z\subset X$. By \cite[Proposition 5.3]{AP1}, if $\D$ is a unitary  $\mu$-semistable divisor, then its combinatorial type $(\E,D)$ is a $\mu$-semistable pseudo-divisor on $\Gamma$.

 By the first and second paragraph of the proof of \cite[Proposition 5.3]{AP1} and Equation \eqref{eq:Wtil}, for every $V\subset V(\Gamma^\E)$ we have:
 \begin{equation}
 \label{eq:betaWZW}
 \beta_\D(X_V)=\beta_D(V)\geq\beta_{\E,D}(V\cap V(\Gamma)), \text{ with equality iff } V=\widetilde{V\cap V(\Gamma)},
 \end{equation}
  and $\beta_\D(Z)\geq \beta_\D(X_{Z\cap V(\Gamma^\E)})$,
  for every subcurve $Z\subset X$.
  We deduce that
 \begin{equation}
 \label{eq:betaXGamma}
\beta_\D(Z)\geq \beta_{\E,D}(Z\cap V(\Gamma)),     
 \end{equation}
for every subcurve $Z\subset X$.

Consider two degree-$d$ divisors $\D_1$ and $\D_2$ on $X$ of combinatorial types $(\E_1,D_1)$ and $(\E_1,D_2)$. We will often have to decide whether or not they are equivalent. This motivates the following definitions.

\begin{Def}
Let $X$ be a tropical curve with a fixed model $\Gamma$.
 A \emph{difference divisor} on $X$ is a divisor $\D$ which can be written as $\D_1-\D_2$, for $\D_1$ and $\D_2$ unitary divisors on $X$ of the same degree.  Equivalently,  a difference divisor is a degree-$0$ divisor of type
   \[
   \D=\D_0+\sum_{e\in \E_1}p_e-\sum_{e\in \E_2}q_e
   \]
   where $\D_0$ is supported on $V(\Gamma)$, the sets $\E_1,\E_2$ are subsets of $\E(\Gamma)$, and $p_e,q_e\in e^\circ$.
\end{Def}

An important case is that in which the divisors defining a difference divisor have the same combinatorial type.

\begin{Def}
 An $\E$-\emph{divisor} is the difference divisor of two unitary  divisors  $\D_1$ and $\D_2$ with the same combinatorial type $(\E,D)$.
Equivalently, a unitary divisor  
is a divisor of type
\[
\D=\sum_{e\in \E} (p_e-q_e),
\]
   where $p_e,q_e\in e^\circ$ are, possibly equal, points in the interior of $e$.
\end{Def}

We are interested in properties of principal $\E$-divisors. Before, we need a lemma.

\begin{Lem}
\label{lem:fmin}
Let $X$ be a tropical curve with model $\Gamma$.
 Let $\D$ be a principal difference divisor on $X$, and write $\D=\Div(f)$, for some rational function $f$ on $X$. Let $Z$ be the subcurve of $X$ where $f$ attains its minimum and set $V=V(\Gamma)\cap Z$. Then $V$ is nonempty and no connected component of $Z$ is contained in the interior of any edge of $\Gamma$. Moreover, $X_V\subset Z$ and $X_{V^c}\subset Z^c$.
\end{Lem}

\begin{proof}
We argue that no connected component of $Z^c$ is contained in the interior of any edge of $\Gamma$. Assume, by contradiction, that $Z^c$ has such a connected component. Let $Y$ be the locus where $f$ attains its maximum on this component. Then 
\[
1\geq \deg(\D|_Y)=\deg(\Div(f)|_Y)\geq 2,  
\]
where the inequality in the left-hand side follows because $\D$ is a difference divisor. This is a
contradiction. The same argument proves that $Z^c$ contains no edges $e$ whose vertices are in $Z$.

Similarly, no connected component of $Z$ is contained in the interior of any edge. Indeed, if $Y$ is a connected component contained in the interior of some edge, then 
\[
-1\leq \deg(\D|_Y)=\deg(\Div(f)|_Y)\leq -2,  
\]
which is a contradiction. \par
  Therefore, $V$ is nonempty. Moreover, if $e\in E(\Gamma)$ is an edge connecting vertices $v$ and $w$ in $V$, then  $v,w\in Z$. If $e^0\cap Z^c\neq \emptyset$, then $Z^c$ would have a component in the interior of $e$ or it would contain all of $e$, which is not possible by the first part of the proof. This proves that $X_V\subset Z$. Arguing similarly, we have that $X_{V^c}\subset Z^c$.
\end{proof}

Let $X$ be a tropical curve with a model $\Gamma$. 
Let $\E$ be a subset of $E(\Gamma)$ and define $\iota\col \Gamma \to\Gamma':=\Gamma/(E(\Gamma)\setminus \E)$. For each $w\in V(\Gamma')$, define $V_w:=\iota^{-1}(w)$.
%We will choose an orientation for each edge in $\E$ such that, if $w_1,w_2\in V(\Gamma')$ and $e_1,e_2\in E(\Gamma')$ are edges connecting $w_1$ and $w_2$, then $s(e_1)=s(e_2)$.

\begin{Lem}
\label{lem:Eunitary}
Let $X$ be a tropical curve with model $\Gamma$.
Let $\D=\sum_{e\in \E}(p_e-q_e)$ be a principal $\E$-divisor on $X$, and write $\D=\Div(f)$, for a rational function $f$ on $X$. For every $e\in \E$ with $p_e\ne q_e$, consider the orientation on $e$ induced by $\overrightarrow{q_ep_e}$. Then:
\begin{enumerate}
    \item[(1)] the slope of $f$ is $0$ everywhere, except on the segments $\overrightarrow{q_ep_e}$, where it is $1$.
    \item[(2)] 
    for every edge $e\in E(\Gamma)$, we have $\ell(\ol{q_ep_e})=f(t(e))-f(s(e))$.
    \item[(3)] If $e_1$ and $e_2$ are edges connecting $V_{w_1}$ and $V_{w_2}$, with $w_1\neq w_2$, then $\ell(\ol{q_{e_1}p_{e_1}})=\ell(\ol{q_{e_2}p_{e_2}})$ and  $\iota(t(e_1))=\iota(t(e_2))$.
    \item[(4)] if $e$ is an edge connecting $V_w$ with itself, then $p_e=q_e$.
    \item[(5)]
    If $\Gamma_\E$ is connected, then $\D=0$.
     \end{enumerate}
\end{Lem}

\begin{proof}
   Let $Z$ be the subcurve of $X$ where $f$ attains its minimum. Then, $Z\cap Z^c$ consists of points $q_e\in e^0$ such that $\D(q_e)=-1$, with $e$ running through a subset of $\E$.  Define $V=Z\cap V(\Gamma)$. By Lemma \ref{lem:fmin}, we have that $V\neq\emptyset$ and $X_V\subset Z$. Moreover, $E(V,V^c)\subset \E$. In particular, $V=\bigcup V_{w_i}$, and $p_e=q_e$ for every edge $e\in \E$, such that $e$ connects each $V_{w_i}$ with itself. \par
   
   Since the slope of $f$ is zero on $Z$ and $\D(q_e)=-1$, for each $e\in E(V,V^c)$, the slope of $f$ on the segment $\overrightarrow{q_ep_e}$ must be $1$. Since $\D(p_e)=1$, we have that $f$ has slope $0$ on $e\setminus\ol{p_eq_e}$.
   
   Consider the subcurve $X_{V^c}$. We know that $f$ has slope $0$ on $e\setminus\ol{p_eq_e}$, for every $e\in E(V,V^c)$, hence  $\Div(f)|_{X_{V^c}}=\Div(f|_{X_{V^c}})$. Therefore,  $\D|_{X_{V^c}}=\Div(f|_{X_{V^c}})$, which is a principal $(\E\cap E(V^c))$-unitary divisor on $X_{V^c}$. So item (1) follows by induction. Note that item (2) readily follows from item (1).\par
   
   To prove item (3), let $\rho\col X\to Y$ be the contraction of all subcurves $X_{V_w}$, for $w\in V(\Gamma')$ (note that the graph underlying a model of $Y$ is obtained contracting $\Gamma'$ at its loops). By items (1) and (2), the function $f$ is constant on each $X_{V_w}$, and so $f$  induces a rational function $\widehat{f}$ on $Y$. Moreover, for every edge $e$ connecting vertices $w_1$ and $w_2$ of $\Gamma'$,  
   \[
   \ell(\ol{q_ep_e})=\widehat{f}(w_2)-\widehat{f}(w_1)
   \]
   (assuming that $w_2=t(e)$). This proves item (3).
   
   Items (4) and (5) are consequences of the previous items.
\end{proof}

%\begin{Def}
%Given a tropical curve $X$, a model $\Gamma$ of $X$ and a bond $C\subset %E(\Gamma)$, we say that a rational function on $X$ is a \emph{$C$-rational %function} if, for every $e\in C$, there exists oriented segments %$\overrightarrow{p_eq_e}\subset e $ such that:
%\begin{enumerate}
%\item the points $p_e$ belongs  to the same connected component of %$X\setminus\bigcup_{e\in C} \ol{p_eq_e}^\circ$.
%\item the length of the segment $\ol{p_eq_e}$ does not depend on $e$.
%\item the rational function $f$ has slope $1$ on each of these oriented %segments. 
%\item the rational function $f$ has slope $0$ everywhere else.
%\end{enumerate}
%Notice that $\Div(f)=D_Y$, where $Y\subset X$ is a subcurve with %$\Out_Y(\Gamma)=C$ and $D_Y$ is a chip-firing divisor emanating from $Y$.
%\end{Def}

\section{Special polytopes and cones}

%Let $X$ be a tropical curve with a model $\Gamma$, and $(\E,D)$ be a pseudo-divisor on $\Gamma$.
We define certain polytopes and cones that will play an important role later in defining some polyhedral decomposition of the tropical Jacobian, and a universal tropical Jacobian.  

 Let $X$ be a tropical curve with oriented model $\Gamma$ and length function $\ell$. In what follows, given a subset $\E\in E(\Gamma)$, we denote by $u_e$ a vector of the canonical basis of $\R^\E$ and by  $x_e$ the coordinates of $\R^\E$, for $e\in \E$.  For each  pseudo-divisor $(\E,D)$ of degree $d$ on $\Gamma$, we define the polytope
 \[
K_{\E,D}(X)=\prod_{e\in \E} [0,\ell(e)] \subset \mathbb{R}^{\E}.
\]
Note that
the interior $K^\circ_{\E,D}(X)$ of $K_{\E,D}(X)$ parametrizes unitary divisors on $X$ with combinatorial type $(\E,D)$. Sometimes, we will identify a unitary divisor with combinatorial type $(\E,D)$ and the corresponding point in $K^\circ_{\E,D}(X)$.

Let $p_0$ be a point in $X$.
There exists a linear map $K_{\E,D}(X)\to \Omega(X)^\vee$,
whose composition with the quotient $\Omega(X)^\vee\ra J(X)$ gives rise to a map 
\begin{equation}\label{eq:Abelmap}
K_{\E,D}(X)\ra  J(X),
\end{equation}
called the \emph{Abel map}. The Abel map $K_{\E,D}(X)\to J(X)$ takes each divisor $\D\in K_{\E,D}(X)$ to the equivalence class of $\D-dp_0$ and it is continuous   (see \cite[Theorem 4.1]{BF} and the proof of \cite[Theorem 5.10]{AP1} for more details).

Now, we define another interesting polytope associated to a pseudo-divisor $(\E,D)$.  Consider the subspace $L_\E\subset \mathbb{R}^\E$ generated by the vectors
 \[
 u_V=\underset{s(e)\in V,t(e)\in V^c}{\sum_{e\in \E}} u_{e}-\underset{t(e)\in V,s(e)\in V^c}{\sum_{e\in \E}}  u_{e}
 \]
 when $V$ runs through all subsets of $V(\Gamma)$ such that $E(V,V^c)\subset \E$ (recall that $u_e$ are the vectors of the canonical basis of $\R^\E$). Equivalently, $L_{\E}$ is given by the equations
 \begin{equation}\label{eq:gammae}
 \sum_{e\in \gamma\cap \E} \gamma(e)x_e=0,
 \end{equation}
 where $\gamma$ runs through all cycles of $\Gamma$ that intersect $\E$ (recall Equation \eqref{eq:gammadef} and that  $x_e$ are the coordinates of $\R^\E$). Note that if $\Gamma_\E$ is connected, then $L_\E=0$. We also note that $\dim(L_\E)=b_0(\Gamma_\E)-b_0(\Gamma)$, hence $\dim(P_{\E,D})=\rk(\E,D)$.
 
  Given a pseudo-divisor $(\E,D)$ on the model $\Gamma$ of the tropical curve $X$,
we consider the linear quotient linear map $T_{\E}\col \mathbb{R}^{\E}\to \mathbb{R}^{\E}/L_\E$. 
%where $V_\E$ is generate by  \par
We define the polytope
\[
P_{\E,D}(X)=T_\E(K_{\E,D}(X)).
\]
Note that if $(\E,D)$ is simple, i.e., if $\E$ is non-disconnecting, then 
\begin{equation}
\label{eq:Ksimple}    
P_{\E,D}(X)=K_{\E,D}(X).
\end{equation}
 
 \begin{Exa}\label{exa:KE}
  Let $X$ be a tropical curve with model $\Gamma$ as in  Figure \ref{fig:unitarydiv}. Let $\ell_1,\ell_2,\ell_3$ be the lengths of the edges of $X$. Let $(\E,D)$ be a pseudo-divisor such that $\E=E(\Gamma)$. We have $K_{\E,D}(X)=[0,\ell_1]\times [0,\ell_2]\times [0,\ell_3]\subset \R^\E=\R^3$. The subspace $L_\E\subset \R^3$ is generated by $u_{e_1}+u_{e_2}+u_{e_3}$ or, equivalently, $L_\E$ is given by the equations $x_{e_1}=x_{e_2}=x_{e_3}$. We see that $P_{\E,D}(X)\subset \R^\E/L_\E$ is a hexagon as in Figure \ref{fig:polyjac}. 
  Two unitary divisors $\D_1$ and $\D_2$ with combinatorial type $(\E,D)$ are equivalent if and only if $\D_1-\D_2=\sum_{e\in E(\Gamma)} (p_e-q_e)$, with $p_e,q_e\in e^0$ such that:
  \begin{itemize}
      \item all the $q_e$ lie in the same connected component after separating $X$ at all points $p_e$ in the edge $e$;
      \item
    there is $r\in\R$ such that $\ell(\ol{q_ep_e})=r$ for every $e\in E(\Gamma)$.
  \end{itemize}  These conditions hold if and only if $u_{\D_1}-u_{\D_2}\in L_\E$, where $u_{\D_1},u_{\D_2}$ are the point of $K_{\E,D}$ corresponding to $\D_1,\D_2$.
  \begin{figure}[h]
\begin{tikzpicture}[scale=2.25]
\begin{scope}[shift={(0,0)}]
\draw (0,0) to [out=30, in=150] (1,0);
\draw (0,0) to (1,0);
\draw (0,0) to [out=-30, in=-150] (1,0);
\draw[fill] (0,0) circle [radius=0.02];
\draw[fill] (1,0) circle [radius=0.02];
%\draw[fill] (0.5,0.144) circle [radius=0.02];
%\draw[fill] (0.5,0) circle [radius=0.02];
%\draw[fill] (0.5,-0.15) circle [radius=0.02];
%\node at (-0.07,0) {1};
%\node at (1.05,0) {1};
%\node at (0.55,0.23) {-1};
%\node at (0.55,0.06) {-1};
%\node at (0.55,-0.22) {-1};
\end{scope}
\end{tikzpicture}
\caption{The model a tropical curve}
\label{fig:unitarydiv}
\end{figure}
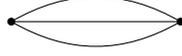
  \end{Exa}

 \begin{Lem}
 \label{lem:DLE}
 Let $X$ be a tropical curve with oriented model $\Gamma$. 
 Let $\D=\sum_{e\in \E}(p_e-q_e)$ be an $\E$-divisor on $X$, for some $\E\subset E(\Gamma)$, and $u_\D$ be the vector in $\mathbb{R}^\E$ given by
 \[
 u_\D=(\varepsilon(e) \ell(\ol{q_ep_e}))_{e\in \E}
 \]
 where $\varepsilon(e)$ is $1$ if $e$ and $\ora{q_ep_e}$ have the same orientation, and $-1$ otherwise. Then $u_\D\in L_\E$ if and only if $\D$ is principal.
 \end{Lem}
 \begin{proof}
 If $\D$ is principal, write $\D=\Div(f)$, for some rational function $f$ on $X$.
 Let $\gamma$ be a cycle of $\Gamma$. Then, by Lemma \ref{lem:Eunitary}, we have
 \[
 \sum_{e\in \gamma\cap \E} \gamma(e)\ell(\ol{q_ep_e})=\sum_{e\in \gamma} \gamma(e)\ell(\ol{q_ep_e})=\pm \sum_{e\in \gamma} \gamma(e)(f(t(e))-f(s(e)))=0.
 \]
 
 Conversely, if $u_\D\in L_\E$, we define a rational function $f$ on $X$ as follows.  Let $g\col X\to \mathbb{R}$ be the function such that $g(p)=1$, if $p\in\ol{q_ep_e}$, and $g(p)=0$, otherwise. Choose a vertex $v\in V(\Gamma)$. For each point $p\in X$, let $\eta_p$ be a path from $v$ to $p$. Then $\int_{\eta_p} \eta_p(t)g(t)dt$ does not depend on the path $\eta_p$ by Equation \eqref{eq:gammae}. We set $f(p)=\int_{\eta_p} \eta_p(t)g(t)dt$, for every $p\in X$. Then $f$ is a rational function on $X$ and $\Div(f)=\sum_{e\in \E}(p_e-q_e)=\D$, so $\D$ is principal.
 \end{proof}

\begin{Prop}
\label{prop:parameter}
The interior $P^\circ_{\E,D}(X)$  parametrizes equivalence classes of unitary divisors $\D$ in X with combinatorial type $(\E,D)$.
\end{Prop}
\begin{proof}
If $\D_1$ and $\D_2$ are two equivalent divisors with combinatorial type $(\E,D)$, we have that $\D_1-\D_2$ is a principal $\E$-divisor. By Lemma \ref{lem:DLE}, we get  $u_{\D_1-\D_2}\in L_\E$. Let $u_{\D_1}, u_{\D_2}\in K_{\E,D}(X)$ be the points  associated to $\D_1$ and $\D_2$. Note that  $u_{\D_1}-u_{\D_2}=u_{\D_1-\D_2}$, then $u_{\D_1}-u_{\D_2}\in L_\E$. We deduce that $\D_1$ and $\D_2$ correspond to the same point in $P_{\E,D}(X)$.

Conversely, if $\D_1$ and $\D_2$ correspond to the same point in $P_{\E,D}(X)$, then $u_{\D_1-\D_2}=u_{\D_1}-u_{\D_2}\in L_\E$. Hence, by Lemma \ref{lem:DLE}, we have that $\D_1$ and $\D_2$ are equivalent. 
\end{proof}

  Let $X$ be a tropical curve with model $\Gamma$. 
  If $(\E,D)\geq(\E',D')$ are two pseudo-divisors on $\Gamma$, there is a natural face-inclusion of polytopes $K_{\E',D'}(X)\subset K_{\E,D}(X)$ induced by the inclusion $\R^{\E'}\subset \R^{\E}$. By Equation \eqref{eq:gammae}, $L_{\E'}=L_\E\cap \R^{\E'}$. Hence, there is a natural inclusion 
  \begin{equation}
\label{eq:PED}  
P_{\E',D'}(X)\subset P_{\E,D}(X).
  \end{equation}
  Note that this inclusion is not necessarily a face-inclusion.\par
   By Proposition \ref{prop:parameter}, the Abel map $K_{\E,D}(X)\to J(X)$, defined in Equation \eqref{eq:Abelmap}, factors through a continuous map
  \begin{equation}
\label{eq:Pcont}
P_{\E,D}(X)\to J(X).
  \end{equation}

 So far, we constructed polytopes that parametrize equivalence classes of unitary divisors on a single tropical curve. We can construct a more universal parameter space in the following way.\par
  Let $\Gamma$ be a graph and $\E$ a subset of $E(\Gamma)$. For each subset $V\subset V(\Gamma)$ such that $E(V,V^c)\subset \E$, define  $\alpha\col E(\Gamma^\E)\to \{0,\pm1\}$ as $\alpha_V(e)=1$ (respectively, $-1$), if $e$ is incident to $V$ (respectively, $V^c$) and lies over an edge in $E(V,V^c)$. Otherwise, we define $\alpha_V(e)=0$.
  We define the vector in $ \R^{E(\Gamma^\E)}$: 
  \begin{equation}\label{eq:omegaV}
  w_V=\sum_{e\in E(\Gamma^\E)} \alpha_V(e)u_e.
 \end{equation}
 (Recall that $u_e$ are the vectors of the canonical basis of $\R^\E$.)
 We let $\EL_\E\subset \R^{E(\Gamma^\E)}$ be the subspace generated by all vectors $w_V$, as $V$ runs through all subsets of $V(\Gamma)$ such that $E(V,V^c)\subset\E$. Also, we denote by 
\[
\T_\E\col \R^{E(\Gamma^\E)}\to \R^{E(\Gamma^\E)}/\EL_{\E}
\]
the linear quotient map. 
We have two natural maps 
\begin{align*}
f_\E\col&\R^{E(\Gamma^\E)}\to \R^{E(\Gamma)}
\;\; \text{ and } \;\;
g_\E\col\R^{E(\Gamma^\E)}\to \R^{\E}
\end{align*}
defined as $f_\E(u_{e'})=u_{e}$ where $e$ is the (unique) edge under $e'$, and $g_\E(u_{e'})=u_{e}$ if $e'=e^s$, otherwise $g_\E(u_{e'})=0$. In other words, the map $f_\E$ takes a vector $(x_{e'})_{e'\in E(\Gamma^\E)}$ to $(y_e)_{e\in E(\Gamma)}$,  while $g_\E$ takes $(x_{e'})$ to $(z_{e})_{e\in \E}$, where 
\begin{equation}\label{eq:yz}
y_{e}=\sum_{e'\text{ over }e} x_{e'} 
\;\;\;\; \text{ and } \;\;\;\;
z_{e}=x_{e^s}.
\end{equation}
%while $g_\E$ takes $(x_{e'})$ to $(z_{e})_{e\in \E}$ with
%\begin{equation}\label{eq:z}
%z_{e}=x_{e^s}.
%\end{equation}
One can see that we have an isomorphism: 
\[
(f_\E,g_\E)\col \R^{E(\Gamma^\E)}\to \R^{E(\Gamma)}\times \R^\E.
\]
We denote by $\tau_{(\Gamma,\E,D)}$ the image cone of $\R^{E(\Gamma^\E)}_{\geq0}$ under $(f_\E,g_\E)$:
\[
\tau_{(\Gamma,\E,D)}:=(f_\E,g_\E)(\R^{E(\Gamma^\E)}_{\geq0})\subset \R^{E(\Gamma)}\times \R^\E.
\]

\begin{Rem}
\label{rem:fEgE}
We note that the cone $\tau_{(\Gamma,\E,D)}$ can be described as 
\[
\tau_{(\Gamma,\E,D)}=\{((x_e)_{e\in E(\Gamma)},(z_e)_{e\in \E}); x_e\geq z_e\geq 0 \text{ for } e \in \E \text{ and } x_e\geq0 \text{ for } e\in E(\Gamma)\}.
\]
(Recall that $x_e$ are the coordinates of $\R^\E$.)
Hence the image of the second projection $\pr\col\tau_{(\Gamma,\E,D)}\to \R^{E(\Gamma)}$ is contained in $\R^{E(\Gamma)}_{\geq0}$, and for every $(x_e)_{e\in E(\Gamma)}\in   \R^{E(\Gamma)}_{\geq0}$,
\begin{equation}
\label{eq:prX}
\pr^{-1}((x_e)_{e\in E(\Gamma)})=\{((x_e)_{e\in E(\Gamma)},(z_e)_{e\in \E}); x_e\geq z_e\geq 0\}.
\end{equation}
So, if $X$ is a tropical curve identified with the point $(\ell(e))_{e\in E(\Gamma)}\in\R^{E(\Gamma)}_{>0}$, where $\Gamma$ and $\ell$ are a model and the length function of $X$ (recall that $\R^{E(\Gamma)}_{>0}$ parametrizes tropical curves $X$ with model $\Gamma$),  we may rephrase Equation \eqref{eq:prX}  as:
\begin{equation}\label{eq:prK}
    \pr^{-1}(X)=\{X\}\times K_{\E,D}(X).
\end{equation}
\end{Rem}

\begin{Lem}
\label{lem:VEker}
Let $\Gamma$ be a graph and $\E$ a subset of $E(\Gamma)$. We have that $(f_\E,g_\E)(\EL_\E)=\{0\}\times L_\E$. In particular, the induced map:
\[
\ol{(f_\E,g_\E)}\col \R^{E(\Gamma^\E)}/\EL_\E\to \R^{E(\Gamma)}\times \R^{\E}/L_\E
\]
is an isomorphism.
\end{Lem}
\begin{proof}
For every subset $V\subset V(\Gamma)$ such that $E(V,V^c)\subset \E$, we have $(f_\E,g_\E)(w_V)=(0,u_V)$, by definition of $w_V$ and $u_V$. The result follows.
\end{proof}

For a graph $\Gamma$ and a pseudo-divisor $(\E,D)$ on $\Gamma$, we define the cone:
\[
 \sigma_{(\Gamma,\E,D)}=\T_\E(\R^{E(\Gamma^\E)}_{\geq0})\subset \R^{E(\Gamma^\E)}/\EL_\E.
\]
By Lemma \ref{lem:VEker}, we have a natural (projection) map $\pi_\E\col \sigma_{(\Gamma,\E,D)}\to \R_{\geq0}^{E(\Gamma)}$.
 The following diagram is commutative:
\begin{equation}
\label{eq:tausigma}
\begin{tikzcd}
\R^{E(\Gamma^\E)}_{\geq0} \ar[r,"\T_\E"] \ar[dd,"{(f_\E,g_\E)}"] \ar[rrrdd,bend left=40, "f_\E"]&     \sigma_{(\Gamma,\E,D)} \ar[dd,"\ol{(f_\E,g_\E)}"] \ar[rddr, "\pi_\E"]\\
\\
\tau_{(\Gamma,\E,D)} \ar[r,"{(\Id,T_\E)}"] \ar[rrr, bend right, "\pr"] &  (\Id,T_\E)(\tau_{\Gamma,\E,D}) \ar[rr,"\ol{\pr}"] & &  \R^{E(\Gamma)}_{\geq0}\\
\end{tikzcd}
\end{equation}
where $\ol{\pr}\col \R^{E(\Gamma)}\times \R^\E/L_\E\ra \R^{E(\Gamma)}$ is the projection on the second factor.
Since $\R^{E(\Gamma^\E)}_{\geq0}$ is isomorphic to $\tau_{(\Gamma,\E,D)}$ (via $(f_\E,g_\E)$), it follows that $\sigma_{(\Gamma,\E,D)}$ is isomorphic to $(\Id,T_\E)(\tau_{\Gamma,\E,D})$.

\begin{Prop}
\label{prop:parameter1}
For every tropical curve $X\in \R^{E(\Gamma)}_{>0}$, there is an isomorphism of polytopes $\pi_\E^{-1}(X)\cong P_{\E,D}(X)$. In particular, 
the open cone $\sigma_{(\Gamma,\E,D)}^\circ$ parametrizes equivalence classes of pairs $(X,\D)$, where $X$ is a tropical curve with model $\Gamma$ and $\D$ is a unitary divisor with combinatorial type $(\E,D)$, and two pairs $(X_1,D_1)$ and $(X_2,D_2)$ are equivalent if $X_1=X_2$ and $D_1$ and $D_2$ are linearly equivalent.
\end{Prop}
 \begin{proof}
 Consider Diagram \eqref{eq:tausigma}. 
Since the map $\overline{(f_\E,g_E)}$ is an isomorphism, we have that $\pi_\E^{-1}(X)$ and $\ol{\pr}^{-1}(X)$ are isomorphic as polytopes. By Equation \eqref{eq:prK} we have $\pr^{-1}(X)=\{X\}\times K_{\E,D}(X)$, and therefore:
\[
\ol{\pr}^{-1}(X)= (\Id,T_\E)(\{X\}\times K_{\E,D}(X))=\{X\}\times P_{\E,D}(X).
\]
Hence, $\pi_\E^{-1}(X)$ is isomorphic to $P_{\E,D}(X)$, as required.\par

The last statement   follows from the first statement and from  Proposition \ref{prop:parameter}.
 \end{proof}

Note that given a specialization of triples $(\Gamma_1,\E_1,D_1)\geq (\Gamma_2,\E_2,D_2)$, then $\EL_{\E_2}=\EL_{\E_1}\cap \R^{E(\Gamma_2^{\E_2})}$, and hence we have a natural inclusion
\begin{equation}
\label{eq:sigmaED}
\sigma_{(\Gamma_2,\E_2,D_2)}\subset \sigma_{(\Gamma_1,\E_1,D_1)}.
\end{equation}

\begin{Prop}\label{prop:face1}
Let $\iota\col \Gamma_1\to \Gamma_2$ be a specialization of graphs. Let $(\E_1,D_1)$ be a pseudo-divisor on $\Gamma_1$, and set $(\E_2,D_2)=\iota_*(\E_1,D_1)$. Then $\sigma_{(\Gamma_2,\E_2,D_2)}$ is a face of $\sigma_{(\Gamma_1,\E_1,D_1)}$.
\end{Prop}
\begin{proof}
We know that  $\R^{E(\Gamma_2^{\E_2})}_{\geq0}$ is naturally a face of $\R^{E(\Gamma_1^{\E_1})}_{\geq0}$ and there is an inclusion $\sigma_{(\Gamma_2,\E_2,D_2)}\subset \sigma_{(\Gamma_1,\E_1,D_1)}$ by Equation \eqref{eq:sigmaED}. 
Let $\E\subset E(\Gamma_1)$ be the set of edges contracted by $\iota$. Let $H$ be the hyperplane in $\R^{E(\Gamma_1^{\E_1})}$ defined as $\sum_{e} x_e=0$, where the sum runs through all edges in $E(\Gamma_1^{\E_1})$ that lie over an edge in $\E$. Note that $\EL_{\E_1}\subset H$ because, if $e_1$ and $e_2$ are the two edges lying over some edge $e'\in \E\cap \E_1$, then the vectors $u_{e_1}$ and $u_{e_2}$ appears with opposite signs in $w_V$, for every $V$ (see Equation \eqref{eq:omegaV}).
  Since $H\cap \R^{E(\Gamma_1^{\E_1})}_{\geq0}=\R^{E(\Gamma_2^{\E_2})}_{\geq0}$, and $\R^{E(\Gamma_1^{\E_1})}_{\geq0}$ is contained in a single half-space defined by $H$, it follows that 
  \[
  (H/\EL_{\E_1})\cap \sigma_{(\Gamma_1,\E_1,D_1)}=\sigma_{(\Gamma_2,\E_2,D_2)}
  \]
and $\sigma_{(\Gamma_1,\E_1,D_1)}$ is contained in a single semi-space defined by $H/\EL_{\E_1}$. Therefore, $\sigma_{(\Gamma_2,\E_2,D_2)}$ is a face of $\sigma_{(\Gamma_1,\E_1,D_1)}$, as required.
\end{proof}

\subsection{Universal tropical Jacobians: a first attempt}\label{sec:attempt}

 As mentioned in the introduction, the category of generalized cone complexes is the right category to construct tropical moduli space. 
We will construct two spaces over $M_g^\trop$ which come close to being universal tropical Jacobians. Throughout, we denote by $M_g^{\trop}$ the moduli space of tropical curves. For more details on $M_g^{\trop}$, we refer to  \cite[Section 2]{M},  \cite[Sections 2.1 and 3.2]{BMV}, \cite[Section 3]{Caporaso1}, \cite[Section 3]{Caporaso}, \cite[Section 4]{ACP}.\par 

First of all, we define the generalized cone complex 
\[
J^{\semi,\trop}_{\mu,g}=\underset{\longrightarrow}{\lim}\sigma_{(\Gamma,\E,D)}
\]
where $\Gamma$ is a stable weighted graph of genus $g$ and $(\E,D)$ is a simple $\mu$-semistable pseudo-divisor on $\Gamma$. Since $(\E,D)$ is simple, it follows that $\sigma_{(\Gamma,\E,D)}=\R^{E(\Gamma^\E)}_{\geq0}$ and each specialization of $(\Gamma,\E,D)$ is also simple. Therefore $J^{\semi,\trop}_{\mu,g}$ is indeed a generalized cone complex. The space $J^{\semi,\trop}_{\mu,g}$ parametrizes equivalence classes of pairs $(X,\D)$, where $X$ is a stable tropical curve of genus $g$ and $\D$ is a simple $\mu$-semistable divisor on $X$.
  The problem with $J^{\semi,\trop}_{\mu,g}$ is that it is possible to have two simple $\mu$-semistable divisors $\D_1$ and $\D_2$, with different combinatorial type, that are linearly equivalent. In a way, the cone complex $J^{\semi,\trop}_{\mu,g}$ has too many points. In Section \ref{sec:strat}, we will see that the algebro-geometric counterpart of $J^{\semi,\trop}_{\mu,g}$ is a stack over $\ol{\mathcal M}_g$ which is not separated.\par
  
  The second space we can construct is the topological space
  \[
P^{\allsemi,trop}_{\mu,g}=\underset{\longrightarrow}{\lim}\sigma_{(\Gamma,\E,D)}
  \]
where $\Gamma$ is a stable weighted graph of genus $g$ and $(\E,D)$ is a $\mu$-semistable pseudo-divisor on $\Gamma$. By Equation \eqref{eq:sigmaED} we have that $\sigma_{(\Gamma',\E',D')}\subset \sigma_{(\Gamma,\E,D)}$ when $(\Gamma,\E,D)\geq(\Gamma',\E',D')$, hence the limit above is well defined. However, as we will see in Proposition \ref{prop:face3}, the topological space $P^{\allsemi,\trop}_{\mu,g}$ is not a generalized cone complex, in the sense that some inclusions $\sigma_{(\Gamma',\E',D')}\subset \sigma_{(\Gamma,\E,D)}$ are not face morphisms. \par
   This motivates the search for a better-behaved limit, which is given by the notion of $\mu$-polystable divisors, that we will introduce in the next section. It is worth to observe that, as a topological spaces, the universal tropical Jacobian we will construct in Section \ref{sec:Jac} is homeomorphic to $P^{\allsemi,\trop}_{\mu,g}$.

  \section{Polystable divisors}

We now introduce the key definition of polystable divisor on graphs and tropical curves. We will prove that in the equivalence class of a divisor on a tropical curve there is a polystable representative. It turns out that two equivalent polystable divisors have the same combinatorial type and their difference is an $\E$-divisor. This property will be crucial to construct a universal tropical Jacobian over $M_g^{\trop}$.

\subsection{Polystability on graphs}
  
  We begin with the definition of polystability for a pseudo-divisor on a graph.
  
  \begin{Def}
  We say that a pseudo-divisor $(\E,D)$ on a graph $\Gamma$ is $\mu$-\emph{polystable} if 
  $\beta_{\E,D}(V)\geq0$ for every subset $V\subset V(\Gamma)$, with strict inequality if $E(V,V^c)\not\subset \E$ (that is, if $V$ is not the set of vertices of a union of connected components of $\Gamma_\E$).
  Equivalently,  $(\E,D)$ is $\mu$-\emph{polystable} if $D_\E$ is $\mu_\E$-stable on $\Gamma_\E$.
  \end{Def}

Note that every $\mu$-polystable pseudo-divisor is $\mu$-semistable. Moreover, every $\mu$-stable divisor is a $\mu$-polystable pseudo-divisor (with $\E=\emptyset$).
The following result tells us that  polystability is well behaved under contractions of graphs.

\begin{Lem}\label{lem:polyspec}
If $\iota\col \Gamma_1\to \Gamma_2$ is a specialization of graphs and $(\E_1,D_1)$ is a $\mu$-polystable  pseudo-divisor, then $\iota_*(\E_1,D_1)$ is $\iota_*(\mu)$-polystable.
\end{Lem}
\begin{proof}
The proof follows directly by the definition and Lemma \ref{lem:betaspec2}.
\end{proof}

Given a graph $\Gamma$ and a degree-$d$ polarization $\mu$ on $\Gamma$, we define $\ps_\mu(\Gamma)$ as the subposets of $\pd(\Gamma)$ consisting  of $\mu$-polystable pseudo-divisors on $\Gamma$.  
We have natural inclusions ($v_0\in V(\Gamma)$ is any vertex of $\Gamma$):
  \[
  \qs_{v_0,\mu}(\Gamma)\subset \sst_\mu(\Gamma)
  \quad\text{ and } \quad
  \ps_\mu(\Gamma)\subset \sst_\mu(\Gamma).
  \]
   For a nondegenerate polarization $\mu$, all the above inclusions are equalities.

  If $\mu$ is a  genus-$g$ universal polarization, we also define $\PS_{\mu,g}$ as the category whose objects are triples $(\Gamma,\E,D)$ where $\Gamma$ is a genus-$g$ stable weighted graph and $(\E,D)$ is a $\mu$-polystable pseudo-divisor on $\Gamma$, and morphisms are given by specializations. As usual, we let $\ps_{\mu,g}$ be the poset $\PS_{\mu,g}/\sim$, where the relation $\sim$ is  isomorphism. 
  %We have a natural inclusion $\ps_{\mu,g}\subset \sst_{\mu,g}$. 

\begin{Prop}
\label{prop:semispec}
Let $\Gamma$ be a graph and $\mu$ a degree-$d$ polarization on $\Gamma$. Let $(\E,D)$ be a $\mu$-semistable pseudo-divisor on $\Gamma$. Assume that $V\subset V(\Gamma)$ is a subset with $\beta_{\E,D}(V)=0$ and $E(V,V^c)\not\subset \E$. Then, there exists a unique $\mu$-semistable pseudo-divisor $(\E\cup E(V,V^c),D_1)$ on $\Gamma$ such that $(\E\cup E(V,V^c),D_1)>(\E, D)$.
\end{Prop}

\begin{proof}
It is sufficient to consider the case where $\E=\emptyset$. Indeed if $\E\neq\emptyset$, we can just replace $\Gamma$, $\mu$ and $D$ with $\Gamma_\E$, $\mu_\E$ and $D_\E$ (see also Equation \eqref{eq:betaee}).\par 
From now on, assume that $\E=\emptyset$. Let $D_1$ be the divisor on $\Gamma^{E(V,V^c)}$ defined as 
\begin{equation}\label{eq:D1}
D_1(v)=\begin{cases}
D(v)+\text{val}_{E(V,V^c)}(v)&\text{ if $v \in V$,}\\
D(v)&\text{ if $v \notin V$,}\\
-1 &\text{ if $v$ is exceptional.}
\end{cases}
\end{equation}
It is clear that $(E(V,V^c),D_1)> (\emptyset, D)$. 

We now argue that $(E(V,V^c),D_1)$ is $\mu$-semistable. 
   Let $W$ be a subset of $V(\Gamma_{E(V,V^c)})=V(\Gamma)$. Define $W_1=W\cap V$ and $W_2=W\cap V^c$.
 Note that there are no edges in $\Gamma_{E(V,V^c)}$ connecting $W_1$ and $W_2$. Hence it follows from Equation \eqref{eq:betasum} that
 \begin{equation}\label{eq:somma}
 \beta_{E(V,V^c),D_1}(W)=\beta_{E(V.V^c),D_1}(W_1)+\beta_{E(V,V^c),D_1}(W_2).
 \end{equation}
 We have $\beta_{E(V,V^c),D_1}(W_1)= \beta_{D}(W_1)\geq0$ and
\begin{align*}
\beta_{E(V,V^c),D_1}(W_2)&=\beta_D(W_2)-|E(V,W_2)|\\
                     &=\beta_D(V\cup W_2)-\beta_D(V)\\
                     &=\beta_D(V\cup W_2)\geq 0.
\end{align*}
By Equation \eqref{eq:somma}, this readily implies that $(E(V,V^c),D_1)$ is $\mu$-semistable. 

 To prove uniqueness, just note that if $(E(V,V^c),D')$ is another pseudo-divisor with $(E(V,V^c),D')> (\emptyset, D)$, then $\deg(D'|_V)< \deg(D_1|_V)$, and hence  
  \[
  \beta_{E(V,V^c),D'}(V)<\beta_{E(V,V^c),D_1}(V)=0,
  \]
  which means that $(E(V,V^c),D')$ is not $\mu$-semistable.
\end{proof}

Let $\Gamma$ be a graph. 
By definition, if $(\E,D)$ is a $\mu$-semistable pseudo-divisor on $\Gamma$ such that $\beta_{\E,D}(V)>0$ for every subset $V\subset V(\Gamma)$ such that $E(V,V^c)\not\subset \E$, then $(\E,D)$ is $\mu$-polystable. Therefore, by Proposition \ref{prop:semispec}, starting from a $\mu$-semistable pseudo-divisor $(\E,D)$, one can construct a sequence of specializations
\begin{equation}
\label{eq:seqsemi}
(\E_k,D_k) > (\E_{k-1},D_{k-1}) >\ldots >  (\E_0,D_0)=(\E,D)
\end{equation}
of $\mu$-semistable divisors, where $(\E_k,D_k)$ is $\mu$-polystable and $(\E_j,D_j)$ is not $\mu$-polystable for $j<k$. In the next result, we show that the polystable divisor $(\E_k,D_k)$ is uniquely determined.

\begin{Prop}\label{prop:minimal}
Let $\Gamma$ be a graph and $\mu$ a degree-$d$ polarization on $\Gamma$. Let $(\E,D)$ be a $\mu$-semistable pseudo-divisor. Then there exists a unique minimal $\mu$-polystable pseudo-divisor $(\E',D')$ on $\Gamma$ such that $(\E',D')\geq  (\E,D)$.
\end{Prop}

\begin{proof}
If $(\E,D)$ is $\mu$-polystable, just take $(\E',D')=(\E,D)$. Otherwise, there is a subset $V\subset V(\Gamma)$ such that $\beta_{\E,D}(V)=0$ and $E(V,V^c)\not\subset\E$. Applying Proposition \ref{prop:semispec} to $(\E,D)$ and $V$ we  obtain a unique $\mu$-semistable pseudo-divisor $(\E \cup E(V,V^c),D_1)$ such that $(\E \cup E(V,V^c),D_1)>(\E,D)$.

   Assume that $(\ol \E, \ol D)$ is a $\mu$-polystable divisor such that $(\ol \E, \ol D)>(\E,D)$. Let us prove that $(\ol \E, \ol D)\geq(\E\cup E(V,V^c),D_1)$. First, we prove that $E(V,V^c)\subset \ol \E$. Indeed, by Lemma \ref{lem:betaspec1}, we have  $\beta_{\ol \E, \ol D}(V)\leq \beta_{\E,D}(V)=0$, hence $\beta_{\ol \E, \ol D}(V)=0$. Therefore, by the definition of polystability, we get $E(V,V^c)\subset \ol \E$.
  By Lemma \ref{lem:middle}, there is a unique $\mu$-semistable pseudo-divisor $(\E \cup E(V,V^c),\widetilde{D})$ such that $(\ol \E, \ol D)\geq(\E\cup E(V,V^c),\widetilde{D})>(\E,D)$. By Proposition \ref{prop:semispec}, we have that $\widetilde{D}=D_1$. This proves that $(\ol \E, \ol D)>(\E\cup E(V,V^c),D_1)$, as wanted. 
	
	Following the same argument, if 
\[
(\E_k,D_k) > (\E_{k-1},D_{k-1}) >\ldots > \ldots (\E_0,D_0)=(\E,D)
\]
is a sequence as in Equation \eqref{eq:seqsemi}, then $(\ol \E,\ol D)\geq(\E_k,D_k)$. This implies that $(\E_k,D_k)$ is a minimal $\mu$-polystable pseudo-divisor on $\Gamma$ such that $(\E_k,D_k)>(\E,D)$, and that it is the unique one satisfying these properties. 
%and that the construction in Equation \eqref{eq:seqsemi} always have the same final pseudo-divisor.
\end{proof}

    As a consequence of Proposition \ref{prop:minimal}, we obtain an order preserving map:
   \[
   \pol\col \sst_{\mu}(\Gamma)\to \ps_\mu(\Gamma)
   \]
   taking $(\E,D)$ to $\pol(\E,D):=(\E',D')$, where $(\E',D')$ is the minimal $\mu$-polystable pseudo-divisor such that $(\E',D')\ge(\E,D)$. Note that the map $\pol$ is a section of the natural inclusion map $\ps_\mu(\Gamma)\to \sst_\mu(\Gamma)$. Moreover, we have 
   \begin{equation}
  \label{eq:rkm}
    \rk(\pol(\E,D))\geq \rk(\E,D),
   \end{equation}
   because $\pol(\E,D)\geq (\E,D)$. 
   %We also have an induced functor $\pol\col \SST_{\mu,g}\to \PS_{\mu,g}$ that is a left inverse of the inclusion functor $\PS_{\mu,g}\to \SST_{\mu,g}$.\footnote{Tirar, se tirarmos as def. das categorias universais}
   
   %We write $\pol_{qs}\col \qs(\Gamma)\to \ps(\Gamma)$ for  the restriction of $m$ to the poset $\qs(\Gamma)$. \par
   
\begin{Lem}\label{lem:rkstrict}
 Let $(\E_1,D_1)$ and $(\E_2,D_2)$ be two $\mu$-polystable divisors on a graph $\Gamma$ such that $(\E_1,D_1)>(\E_2,D_2)$. Then $\rk(\E_1,D_1)>\rk(\E_2,D_2)$.
\end{Lem}   
\begin{proof}
By Equation \eqref{eq:betaee}, we can assume that $\E_2=\emptyset$, and in this case $D_2$ is stable. Also, it is sufficient to prove the statement for a connected component of $\Gamma$.\par
  Assume that $\rk(\E_1,D_1)=\rk(\emptyset, D_2)=0$. Consider the specialization 
  \[
  \iota\col \Gamma\to\Gamma'=\Gamma/(E(\Gamma)\setminus \E_1).
  \]
  Since $\rk(\E_1,D_1)=0$, the graph $\Gamma'$ is a tree.  By Lemma \ref{lem:polyspec}, the  pseudo-divisors $\iota_*(\E_1,D_1)$ and $(\iota_*(\emptyset, D_2)$ are polystable. So we can assume that $\Gamma=\Gamma'$ is a tree and $\E=E(\Gamma)$.\par
  The specialization $(\E_1,D_1)\to (\emptyset, D_2)$ is induced by a specialization $\iota\col \Gamma^{\E_1}\to \Gamma$. We define an orientation on $\Gamma$ such that $\iota(v_e)=t(e)$ for every $e\in E(\Gamma)$ (recall that $v_e$ is the exceptional vertex of $\Gamma^{\E_1}$ contained in $e$). Since $\Gamma$ is a tree, this orientation is acyclic, so it has a sink $v_0$. We get that $D_2(v_0)=D_1(v_0)-\val_{E(\Gamma)}(v_0)$, then
  \begin{align*}
  \beta_{D_2}(v_0)&=D_2(v_0)-\mu(v_0)+\frac{\val_{E(\Gamma)}(v_0)}{2}\\
  &=D_1(v_0)-\mu(v_0)-\frac{\val_{E(\Gamma)}(v_0)}{2}\\
  &=D_1(v_0)-\mu_{E(\Gamma)}(v_0)\\
  &=\beta_{\E_1,D_1}(v_0)=0,
  \end{align*}  
  which is a contradiction.
\end{proof}
   
 \begin{Prop}
 \label{prop:polyquasi}
 Let $\Gamma$ be a graph, $\mu$ polarization on $\Gamma$ and $v_0$ a vertex of $\Gamma$. If $(\E,D)$ is a $\mu$-polystable pseudo-divisor, then there is a $(v_0,\mu)$-quasistable pseudo-divisor $(\E',D')$ such that $\pol(\E',D')=(\E,D)$ and $\rk(\E',D')=\rk(\E,D)$.
  \end{Prop}
 \begin{proof}
  Consider the specialization:
  \[
  \iota\col \Gamma\to \Gamma'=\Gamma/(E(\Gamma)\setminus\E).
  \]
     Let $T\subset \E$ be a spanning tree of $\Gamma'$. We know, by Equation \eqref{eq:betaee}, that $(\E,D)$ is $\mu$-polystable if and only if $(T,D_{\E\setminus T})$ is a $\mu_{\E\setminus T}$-polystable pseudo-divisor on $\Gamma_{\E\setminus T}$.
   Hence, we can assume that $\E=T$. In this case, $\Gamma'$ is a tree. We will view $\iota(v_0)$ as the root of $\Gamma'$. This gives rise to an orientation $s,t\col \E\to V(\Gamma)$ (pointing away from the root). Moreover, $\rk(\E,D)=0$.\par 
  
  We define the pseudo-divisor $(\emptyset,D')$ on $\Gamma$ as
  \[
  D'(v)=\begin{cases}
   D(v)-1&\text{ if $v=t(e)$ for some $e\in\E$ }\\
   D(v) & \text{ otherwise}.
  \end{cases}
  \]
 
  We have that $\rk(\emptyset, D')=0=\rk(\E,D)$ and $(\E,D)\geq (\emptyset,D')$. By Lemma \ref{lem:rkstrict}, all that is left is to prove is that $D'$ is $(v_0,\mu)$-quasistable. Since $D'$ is the specialization of a $\mu$-polystable pseudo-divisor, it is $\mu$-semistable. Hence, there remains to prove that $\beta_{D'}(V)>0$ for every proper subset $V\subsetneqq V(\Gamma)$ such that $v_0\in V$.
  
    Let $V\subset V(\Gamma)$. For each vertex $w\in V(\Gamma')$, let $V_w\subset V$ be the subset of vertices $v\in V$ such that $\iota(v)=w$. Moreover, let $w_0:=\iota(v_0)$ and, for each $w\in V(\Gamma')$ with $w\ne w_0$, let $e_w$ be the unique edge in $\E$ such that $\iota(t(e))=w$. By Equation \eqref{eq:betasumD}:
\[
\beta_{D'}(V)=\sum_{w\in V(\Gamma')} \beta_{D'}(V_w)-|\E_V |,
\]
 where $\E_V$ is the set:
 \[
 \E_V=\{e\in \E; e \in E(V_{\iota(s(e))}, V_{\iota(t(e))}\}.
  \]
  However,
 \[
 \beta_{D'}(V_w)=\beta_{\E,D}(V_w)+|\E_w|,
 \]
 where $\E_w$ is the set
 \[
 \E_w=\{e \in \E; \iota(s(e))=w\text{ and } e\in E(V_w,V_w^c)\}.
 \]
  Note that the $\E_w$ are pairwise disjoint. Then
 \[
 \beta_{D'}(V)=\sum_{w\in V(\Gamma')} \beta_{\E,D}(V_w) -|\E_V| 
 +\left|\coprod_{w\in V(\Gamma')}\E_w\right|.
 \]
 On the other hand, we have $\E_V\subset \coprod\E_w$. 
 
 Assume now that $\beta_{D'}(V)=0$, for some $V\subset V(\Gamma)$.  Then $\beta_{\E,D}(V_w)=0$, for every $w\in V(\Gamma')$, and $\E_V=\coprod \E_w$. In this case, since $(\E,D)$ is $\mu$-polystable, we get that either $V_w=\iota^{-1}(w)$ or $V_w=\emptyset$. Let 
 \[
 W:=\{w\in V(\Gamma');V_w=\iota^{-1}(w)\}.
 \]
 Note that $w_0\in W$, because $v_0\in V$ and hence $v_0\in V_{w_0}\neq\emptyset$. 
 We claim that $W$ is equal to $V(\Gamma')$. Indeed, if this is not the case, then there is an edge $e\in \E$, such that $\iota(s(e))\in W$ and $\iota(t(e))\notin W$. In other words, $V_{\iota(s(e))}=\iota^{-1}(\iota(s(e)))$ and $V_{\iota(t(e))}=\emptyset$. This means that $e\in \E_{\iota(s(e))}\setminus \E_V$, hence $\E_V\neq \coprod \E_{w}$, a contradiction. 
 Then $W$ is equal to $V(\Gamma')$, and hence $V$ is equal to $V(\Gamma)$, and we are done.
 \end{proof}

   \begin{Exa}\label{exa:polyposet}
   Consider the graph $\Gamma$ in Figure \ref{fig:unitarydiv}. Let $\mu$ be the polarization on $\Gamma$ of degree $-1$ given by $\mu(v)=-1/2$ for every $v\in V(\Gamma)$. Figure \ref{fig:polyposet}  illustrates the poset $\ps_\mu(\Gamma)$ of $\mu$-polystable pseudo-divisors of degree $-1$ on $\Gamma$. 
If $v_0$ is one of the two vertices of $\Gamma$, then the intersection $\ps_\mu(\Gamma)\cap \qs_{v_0,\mu}(\Gamma)$ consists of the 3 pseudo-divisors in Figure \ref{fig:polyposet} of type $(\E,D)$, with $\E\ne E(\Gamma)$. The intersection $\qs_{v_0,\mu}(\Gamma)\setminus\ps_\mu(\Gamma)$ consists of the pseudo-divisors $(\E',D')$, with $|\E'|\in\{0,1,2\}$ and $D'(v_0)=1$. For any one of them, $\pol(\E',D')$ is the polystable pseudo-divisor $(E(\Gamma),\ol D)$ on the left in Figure \ref{fig:polyposet}. Note that $\rk(\E',D')=\rk(E(\Gamma),\ol D)$ if $|\E'|=2$.
\begin{figure}[hb]
\begin{tikzpicture}[scale=2.25]
\begin{scope}[shift={(0,0)}]
\draw (0,0) to [out=30, in=150] (1,0);
\draw (0,0) to (1,0);
\draw (0,0) to [out=-30, in=-150] (1,0);
\draw[fill] (0,0) circle [radius=0.02];
\draw[fill] (1,0) circle [radius=0.02];
\draw[fill] (0.5,0.144) circle [radius=0.02];
\draw[fill] (0.5,0) circle [radius=0.02];
\draw[fill] (0.5,-0.15) circle [radius=0.02];
\node at (-0.07,0) {1};
\node at (1.05,0) {1};
\node at (0.55,0.23) {-1};
\node at (0.55,0.06) {-1};
\node at (0.55,-0.22) {-1};
\end{scope}
\draw[->] (1.2,0) to (1.8, 0);
\draw[->] (3.3,0.1) to (3.8, 0.4);
\draw[->] (3.3,-0.1) to (3.8, -0.4);
\begin{scope}[shift={(0,0)}]
\draw (0+2,0) to [out=30, in=150] (1+2,0);
\draw (0+2,0) to (1+2,0);
\draw (0+2,0) to [out=-30, in=-150] (1+2,0);
\draw[fill] (0+2,0) circle [radius=0.02];
\draw[fill] (1+2,0) circle [radius=0.02];
%\draw[fill] (0.5+2,0.144) circle [radius=0.02];
%\draw[fill] (0.5+2,0) circle [radius=0.02];
\draw[fill] (0.5+2,-0.15) circle [radius=0.02];
\node at (-0.07+2,0) {0};
\node at (1.05+2,0) {0};
%\node at (0.55+2,0.23) {-1};
%\node at (0.55+2,0.06) {-1};
\node at (0.55+2,-0.22) {-1};
\end{scope}
\begin{scope}[shift={(0,0)}]
\draw (0+2+2,0+0.5) to [out=30, in=150] (1+2+2,0+0.5);
\draw (0+2+2,0+0.5) to (1+2+2,0+0.5);
\draw (0+2+2,0+0.5) to [out=-30, in=-150] (1+2+2,0+0.5);
\draw[fill] (0+2+2,0+0.5) circle [radius=0.02];
\draw[fill] (1+2+2,0+0.5) circle [radius=0.02];
%\draw[fill] (0.5+2+2,0.144) circle [radius=0.02];
%\draw[fill] (0.5+2+2,0) circle [radius=0.02];
%\draw[fill] (0.5+2+2,-0.15) circle [radius=0.02];
\node at (-0.07+2+2,0+0.5) {0};
\node at (1.05+2+2,0+0.5) {\;\;\;-1};
%\node at (0.55+2+2,0.23) {-1};
%\node at (0.55+2+2,0.06) {-1};
%\node at (0.55+2+2,-0.22) {-1};
\end{scope}
\begin{scope}[shift={(0,0)}]
\draw (0+2+2,0-0.5) to [out=30, in=150] (1+2+2,0-0.5);
\draw (0+2+2,0-0.5) to (1+2+2,0-0.5);
\draw (0+2+2,0-0.5) to [out=-30, in=-150] (1+2+2,0-0.5);
\draw[fill] (0+2+2,0-0.5) circle [radius=0.02];
\draw[fill] (1+2+2,0-0.5) circle [radius=0.02];
%\draw[fill] (0.5+2+2,0.144) circle [radius=0.02];
%\draw[fill] (0.5+2+2,0) circle [radius=0.02];
%\draw[fill] (0.5+2+2,-0.15) circle [radius=0.02];
\node at (-0.07+2+2,0-0.5) {-1};
\node at (1.05+2+2,0-0.5) {0};
%\node at (0.55+2+2,0.23) {-1};
%\node at (0.55+2+2,0.06) {-1};
%\node at (0.55+2+2,-0.22) {-1};
\end{scope}
\end{tikzpicture}
\caption{The poset of polystable pseudo-divisors.}
\label{fig:polyposet}
\end{figure}
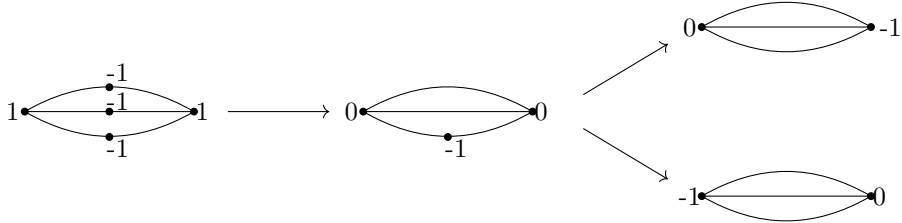
   \end{Exa}

\subsection{Polystability on tropical curves}

We consider polystability for divisors on tropical curves.  
Let $X$ be a tropical curve,  $\mu$ a polarization on $X$ and $\Gamma$ a $\mu$-model of $X$. 

\begin{Def}
A divisor $\D$ on $X$ is $\mu$-\emph{polystable} if there is a $\mu$-polystable pseudo-divisor $(\E,D)$ on $\Gamma$ such that $\D\in K^\circ_{\E,D}(X)$.
\end{Def}

%Given a divisor $\D$ on $X$, we say that $\D$ is \emph{$\mu$-polystable} if its combinatorial type $(\E,D)$ is $\mu$-polystable.

%\begin{Def}   We say that a divisor $D$ of degree $d$ on $X$  is \emph{$\mu$-polystable} if 
%\begin{itemize}\label{def:polydef}
%    \item[(1)] 
%    $\Gamma^D$ is the $\E_D$-subdivision of $\Gamma$;
%    \item[(2)] $D(p)=-1$ for every $p\in V(\Gamma^D)\setminus V(G)$;
%    \item[(3)] $D_{\E_D}$ is $\mu_{\E_D}$-stable on $X_{\E_D}$.
%\end{itemize}
%\end{Def}

% Note that the third condition imply that $\mu_{\E_D}$ is a polarization of degree $d+|\E_D|$ (i.e, of degree $\deg D_{\E_D}$) on $X_{\E_D}$.
 The goal of this section is to prove the following theorem (see also \cite[Proposition 4.4]{CPS} for the same statement).
 
\begin{Thm}
\label{thm:poly}
Let $X$ be a tropical curve and $\mu$ a degree-$d$ polarization on $X$. The following properties hold.
\begin{enumerate}
    \item Every degree-$d$ divisor on $X$ is equivalent to a $\mu$-polystable divisor.
    \item Two equivalent $\mu$-polystable divisors have the same combinatorial type.
\end{enumerate}
\end{Thm}

\begin{Rem}
\label{rem:D1D2}
Although a $\mu$-polystable divisor on a tropical curve is not unique in its equivalence class, Theorem \ref{thm:poly} tells us that the difference of equivalent polystable divisors is well behaved. In fact, let $\D_1$ and $\D_2$ be two equivalent $\mu$-polystable divisors. By item (2) of Theorem \ref{thm:poly}, the divisors $\D_1$ and $\D_2$ have the same combinatorial type $(\E,D)$, hence $\D_1-\D_2$ is a principal divisor of type:
\[
\D_1-\D_2=\sum_{e\in \E} (p_e-q_e),
\]
where $p_e,q_e\in e^0$ are, possibly equal, points in the interior of $e$. 
In other words, $\D_1-\D_2$ is a principal $\E$-divisor. Recall that we described some important properties of $\E$-divisors in Lemmas \ref{lem:Eunitary} and \ref{lem:DLE}. Later on, this will be crucial to  construct a universal tropical Jacobian over the moduli space of tropical curves.
\end{Rem}

Before proving Theorem \ref{thm:poly}, we shall give an example to explain how to convert a divisor into a polystable divisor.

\begin{Exa}\label{exa:convert}
Let $X$ be a tropical curve $X$ as in Example \ref{exa:KE}. Let $X$ be the polarization on $X$ such that $\mu(v)=-1/2$ for every $v\in V(\Gamma)$. Consider the divisor  $\D=v_0-v_1-p$, where $V(\Gamma)=\{v_0,v_1\}$ and $p\in e^\circ$ for some $e\in E(\Gamma)$. Note that $\D$ is $v_0$-quasistable, but not polystable. In Figure \ref{fig:polyequiv} we illustrate how to convert $\D$ into a polystable divisor. 
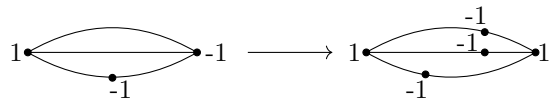
\begin{figure}[h]
\begin{tikzpicture}[scale=2.25]
\begin{scope}[shift={(0,0)}]
\draw (0,0) to [out=30, in=150] (1,0);
\draw (0,0) to (1,0);
\draw (0,0) to [out=-30, in=-150] (1,0);
\draw[fill] (0,0) circle [radius=0.02];
\draw[fill] (1,0) circle [radius=0.02];
\draw[fill] (0.7,0) circle [radius=0.02];
\draw[fill] (0.7,0.12) circle [radius=0.02];
\draw[fill] (0.35,-0.13) circle [radius=0.02];
\node at (-0.07,0) {1};
\node at (1.05,0) {1};
\node at (0.65,0.22) {-1};
\node at (0.6,0.06) {-1};
\node at (0.3,-0.22) {-1};
\end{scope}
\draw[->] (-0.7,0) to (-0.2, 0);
\begin{scope}[shift={(0,0)}]
\draw (0-2,0) to [out=30, in=150] (1-2,0);
\draw (0-2,0) to (1-2,0);
\draw (0-2,0) to [out=-30, in=-150] (1-2,0);
\draw[fill] (0-2,0) circle [radius=0.02];
\draw[fill] (1-2,0) circle [radius=0.02];
%\draw[fill] (0.5+2,0.144) circle [radius=0.02];
%\draw[fill] (0.5+2,0) circle [radius=0.02];
\draw[fill] (0.5-2,-0.15) circle [radius=0.02];
\node at (-0.07-2,0) {1};
\node at (1.05-2,0) {\;\;\;-1};
%\node at (0.55+2,0.23) {-1};
%\node at (0.55+2,0.06) {-1};
\node at (0.55-2,-0.22) {-1};
\end{scope}
\end{tikzpicture}
\caption{Converting a quasistable divisor into a polystable divisor.}
\label{fig:polyequiv}
\end{figure}
\end{Exa}

\begin{proof}[Proof of item (1) of Theorem \ref{thm:poly}]
  By \cite[Theorem 5.6]{AP1}, every divisor on $X$ is equivalent to a $\mu$-semistable unitary divisor. So it suffices to show that every $\mu$-semistable unitary divisor $\D$ on $X$ is equivalent to a $\mu$-polystable divisor on $X$.\par
  
  Let $\Gamma$ be the $\mu$-model of $X$.
  Let $(\E,D)$ be the combinatorial type of $\D$. If $(\E,D)$ is $\mu$-polystable, there is nothing to do. Otherwise there exists $V\subset V(\Gamma)$ such that $E(V,V^c)\not\subset \E$ and $\beta_{\E,D}(V)=0$. We can choose $V$ such that $E(V,V^c)$ is a minimal disconnecting subset. We apply Proposition \ref{prop:semispec} to get a $\mu$-semistable pseudo-divisor $(\E\cup E(V,V^c),D_1)$ such that $(\E\cup E(V,V^c),D_1)>(\E,D)$. 
  
  We claim that there exists a divisor $\D_1$ on $X$ with combinatorial type $(\E\cup E(V,V^c),D_1)$ that is equivalent to $\D$.
  Indeed, for each $e\in E(V,V^c)$ let $v_e,w_e$ be the vertices incident to $e$ such that $v_e\in V$ and $w_e\in V^c$. Also, for each $e\in \E\cap E(V,V^c)$, let $p_e\in e^\circ$ be the point of $e^\circ $ such that $\D(p_e)=-1$, while, if $e\in E(V,V^c)\setminus \E$, let $p_e:=v_e$. Let $\ell$ be the length function of $X$ and define: 
  \[
  r=\min\{\ell(\ol{p_ew_e});e\in E(V,V^c)\}.
  \]
For $e\in E(V,V^c)$, let $q_e\in \ol{p_ew_e}$ be the point such that $\ell(\ol{p_eq_e})=\frac{r}{2}$. Let $f$ be the rational function with slope $1$ on $\ora{q_ep_e}$ and $0$ everywhere else. Note that
 \[
 \Div(f)(p)=
 \begin{cases}
 \val_{E(V)\setminus\E}(v_e) & \text{ if $p=v_e$ for some $e\in E(V,V^c)\setminus \E$},\\
 1 & \text{ $p=p_e$ for some $e\in E(V,V^c)\cap \E$},\\
 -1& \text{ $p=q_e$ for some $e\in E(V,V^c)$},\\
 0& \text{ otherwise.}
  \end{cases}
 \]
Define $\D_1=\D+\Div(f)$.
 Then, comparing to Equation \eqref{eq:D1}, the combinatorial type of $\D_1$ is $(\E\cup E(V,V^c),D_1)$, proving the claim. 
 
 Repeating this process and recalling Equation \eqref{eq:seqsemi} one can prove that there exists a $\mu$-polystable divisor $\D_k$ equivalent to $\D$. 
\end{proof}

%Let $\D$ be a $\mu$-polystable divisor on $X$. We denote by $X^1_{\E_\D},\dots,X^{c_\D}_{\E_\D}$ the connected components of $X_{\E_\D}$, and we set $\D^i_{\E_\D}:=\D|_{X^i_{\E_\D}}$ for $i=1,\dots,c_\D$. We call $(X^1_{\E_\D},\D^1_{\E_\D}),\dots,(X^{c_\D}_{\E_\D},\D^{c_\D}_{\E_\D})$ the \emph{$\D$-factors} and $c_\D$ the \emph{length} of $\D$. 

%\begin{Lem}
%Let $f$ be a rational function on a segment $e=\ol{pq}$ such that $f(p)=f(q)$. Then $\Div(f)|_{e^\circ}\neq R_1-R_2$ for every $R_1,R_2\in %e$ with $R_1\neq R_2$.
%\end{Lem}

Before proving item (2) of Theorem \ref{thm:poly}, we need a lemma.

\begin{Lem}\label{lem:betazero}
 Let $X$ be a tropical curve. Let $\D$ be a $\mu$-polystable divisor on $X$ of combinatorial type $(\E,D)$. Let $Z$ be a tropical subcurve of $X$. Consider $V=Z\cap V(\Gamma)$ and, for every $e\in E(V,V^c)$, let $v_e$ be the endpoint of $e$ with $v_e\in V$. If $V\ne\emptyset$ and $\beta_\D(Z)=0$, then $E(V,V^c)\subset \E$ and $\ol{v_ep_{\D,e}}\subset Z$ for every $e\in E(V,V^c)$.
\end{Lem}

\begin{proof}
If $Z=Z_1\coprod Z_2$, then $\beta_\D(Z)=\beta_\D(Z_1)+\beta_\D(Z_2)$, and the result for $Z$ follows from the result for $Z_1$ and $Z_2$. Therefore, we can assume that $Z$ is connected. By Equation \eqref{eq:betaXGamma}, we have 
\begin{equation}\label{eq:equalzero}
0=\beta_\D(Z)\geq \beta_{\E,D}(V).
\end{equation}
Since $(\E,D)$ is $\mu$-polystable, we have that $\beta_{\E,D}(V)=0$ and so, by the definition of polystability, we deduce that $E(V,V^c)\subset \E$. \par

  Let us prove that
  $\ol{v_ep_{\D,e}}\subset Z$ for every $e \in E(V,V^c)$. Since equality holds in Equation \eqref{eq:equalzero}, we have that $Z\cap V(\Gamma^\E)=\widetilde{Z\cap V(\Gamma)}$ by Equations \eqref{eq:betaWZW} and \eqref{eq:betaXGamma}.
  In particular, $p_{\D,e}\in Z$ for every $e\in \E\cap E(V,V^c)$. Hence,
 \begin{align*}
 \beta_\D(Z\cup \ol{v_ep_{\D,e}})&=\deg(\D|_{Z\cup\ol{v_ep_{\D,e}}})-\mu(Z\cup \ol{v_ep_{\D,e}})+\frac{\delta_{X,Z\cup \ol{v_ep_{\D,e}}}}{2}\\
                                 &=\deg(\D|_Z)-\mu(Z)+\frac{\delta_{X,Z}}{2}-\frac{\delta_{X,Z}-\delta_{X,Z\cup \ol{v_ep_{\D,e}}}}{2}\\
                                 &=\beta_{\D}(Z)-\frac{\delta_{X,Z}-\delta_{X,Z\cup \ol{v_ep_{\D,e}}}}{2}.
 \end{align*}
Since $\beta_{\D}(Z)=0$ and $\beta_\D(Z\cup \ol{v_ep_{\D,e}})\geq0$, we get
\begin{equation}\label{eq:deltaXZ}
\delta_{X,Z}\leq\delta_{X,Z\cup \ol{v_ep_{\D,e}}}.
\end{equation}
This can only happen if $\ol{v_ep_{\D,e}}\subset Z$ (recall that $v_e,p_{\D,e}\in Z$), and in this case, equality holds in \eqref{eq:deltaXZ}.
\end{proof}

%\begin{Thm}\label{thm:ployfactor}
%The factors of equivalent $\mu$-polystable divisors are equal.
%\end{Thm}

\begin{proof}[Proof of item (2) of Theorem \ref{thm:poly}]
  Let $\Gamma$ be the $\mu$-model of $X$.
Let $\D$ and $\D'$ be equivalent $\mu$-polystable divisors on $X$ of combinatorial type $(\E,D)$ and $(\E',D')$. We proceed by induction on the number of connected components of $\Gamma_\E$.
Write $\D'-\D=\Div(f)$, with $f$ a rational function on $X$. Let $Z\subset X$ be the set of points where $f$ attains its minimum.  Define $V=V(\Gamma)\cap Z$. By Lemma \ref{lem:fmin}, we have $X_V\subset Z$, hence
 %\[
 %\{e \in E(\Gamma); \emptyset\ne e\cap Z\subsetneqq e\} =E(V).
 %\]
  in particular $\delta_{X,Z}=|E(V,V^c)|$. For every $e\in E(V,V^c)$, let $v_e,w_e$ be the endpoints of $e$, with $v_e\in V$ and $w_e\in V^c$, so that $v_e\in Z$ and $w_e\notin Z$. Also, define $p_e$ such that $Z\cap e= \overline{v_ep_e}$ (by Lemma \ref{lem:fmin}, such a point $p_e$ exists). We let $k_e>0$ be the slope of $f$ going out of $p_e$ in the direction of $e$. Then
\[
\beta_{\D'}(Z)-\beta_\D(Z)=\deg_{\D'}(Z)-\deg_\D(Z)=\sum_{e \in E(V,V^c)} k_e.
\]
Since a polystable divisor is semistable, we get $k_e=1$ for every $e\in E(V,V^c)$, hence $\beta_\D(Z)=0$ and  $\beta_{\D'}(Z)= |E(V,V^c)|$. By Lemma \ref{lem:betazero} we obtain that 
%$(Z_{\E_D},D|_{\E_D})$ is a union of $D$-factors and $e_{D,Z}\subset Z$ for every $e\in \E_D$. Similarly, using that $\beta_{D'}((Z_\varepsilon)^c,\mu)=0$, we get that $Z^c_{\E_{D'}}$ (if not empty) is a union of $D'$ factor and $e_{D',Z^c}\subset (Z_\varepsilon)^c$ for every $e\in \E_{D'}$. In particular, we obtain 
 \begin{equation}\label{eq:EE'}
 E(V,V^c)\subset \E\cap \E'.
 \end{equation}
and $p_{\D,e}\in Z$ for every $e\in E(V)$. Similarly, if $Z_0$ is some sufficiently small neighborhood of $Z$, then $\beta_{\D'}(Z_0^c)=0$, hence $p_{\D',e}\in Z_0^c$, so that $p_{\D,e}\ne p_{\D',e}$. 
Moreover, since $f$ is constant on $Z$,  it is constant on a neighborhood of $X_V$. Therefore $\D|_{X_V}=\D'|_{X_V}$, and hence,
\begin{equation}
\label{eq:EDV}
(\E,D)|_{\Gamma(V)}=(\E',D')|_{\Gamma(V)},
\end{equation}
where $\Gamma(V)=(V,E(V))$ is the subgraph of $\Gamma$ induced by $V$.
% Hence each connected component of $(Z_{\E_D},D|_{\E_D})$ is a $D$-factor and a $D'$-factor.

If $f$ is constant, then $Z=X$, and we are done.  So, assume that $f$ is not constant, and hence $Z\ne X$.

We claim that the slope of $f$ is $1$ over $\overline{p_{\D,e}p_{\D',e}}$ and $0$ over $\overline{v_e p_{\D,e}}\cup \overline{ p_{\D',e}w_e}$, for every $e\in E(V,V^c)$. Indeed for every $e\in E(V,V)$, the point $p_{\D,e}$ (respectively, $p_{\D',e}$) is the unique point of $e^0$ over which $\D$ (respectively, $\D'$) is different from $0$.  Since the slope of $f$ changes from $0$ to $1$ at the point $p_e$ of $e^\circ$, which is different from $p_{\D',e}$, we get $\D(p_e)\ne 0$, and hence   $p_e=p_{\D,e}$. Thus $f$ has slope zero on $\overline{v_ep_e}$ and $1$ on $\overline{p_eq_e}$, for some $q_e\in e\cap Z^c$. 
%On the other hand, $p_{D',e}$ is the unique point in $e^0$ over which $D'$ is different form $0$. Thus 
The condition $\D'-\D=\Div(f)$ implies $q_e=p_{\D',e}$, and forces $f$ to have slope zero over  $\overline{q_ew_e}$ and hence, over the whole set   $\ol{v_ep_e}\cup \ol{q_ew_e}$, proving the claim.  

Now take $X_{V^c}$.  Consider the rational function $\wh f:=f|_{X_{V^c}}$ and the polarization $\wh\mu:=\mu_{E(V)}|_{X_{V^c}}$ on $X_{V^c}$. Define the divisors on $X_{V^c}$:
\[
\wh \D:=\D|_{X_{V^c}}
\;\; \text{ and } \;\;
\wh \D':=\D'|_{X_{V^c}}.
\]
Since $f$ has slope $0$ on $\ol{p_{\D',e}w_e}$ for every $e\in E(V)$, we deduce that:
\[
\Div(\wh f)=\Div(f)|_{X_{V^c}}=\wh \D'-\wh \D.
\]
 Moreover, $\wh \D$ and $\wh \D'$ have combinatorial type $(\E, D)|_{\Gamma({V^c})}$ and $(\E',D')|_{{\Gamma({V^c})}}$, which are  $\wh\mu$-polystable pseudo-divisors, by Equation \eqref{eq:betaee}. By inductive hypothesis, 
 \begin{equation}
 \label{eq:EDVc}
 (\E,D)|_{{\Gamma({V^c})}}=(\E',D')|_{{\Gamma({V^c})}}.
\end{equation}
Combining Equations \eqref{eq:EE'}, \eqref{eq:EDV} and \eqref{eq:EDVc}, we get $(\E,D)=(\E',D')$, as required.
\end{proof}

\section{A universal tropical Jacobian over $M_g^{\trop}$}\label{sec:Jac}

\subsection{A polyhedral decomposition of the tropical Jacobian}

In this section, we will construct a polyhedral decomposition of the Jacobian of a tropical curve by means of $\mu$-polystable divisors, and we will compare it with known decompositions. 

In this section, we denote by $X$ a tropical curve with a polarization $\mu$ and $\mu$-model $\Gamma$. We already defined the polytopes $P_{\E,D}(X)$ for a pseudo-divisor $(\E,D)$ on $\Gamma$. Recall that if $(\E',D')\geq(\E,D)$, then there are  natural inclusions $K_{\E,D}(X)\subset K_{\E',D'}(X)$ and $P_{\E,D}(X)\subset P_{\E',D'}(X)$ (recall Equation \eqref{eq:PED}). We will prove that these polytopes glue nicely when we consider $\mu$-polystable pseudo-divisors. 

\begin{Prop}
\label{prop:semipolyint}
Let $(\E,D)$ be a $\mu$-semistable pseudo-divisor on $\Gamma$. If we let $(\E',D')=\pol(\E,D)$, then
\[
P_{\E,D}^\circ(X)\subset P_{\E',D'}^\circ(X).
\]
\end{Prop}

\begin{proof}
If $(\E,D)$ is $\mu$-polystable, there is nothing to do. Otherwise, there is $V\subset V(\Gamma)$ such that $\beta_{\E,D}(V)=0$ and $E(V,V^c)\not\subset \E$. Define $\E_1=\E\cup E(V,V^c)$ and let $(\E_1,D_1)$ be the pseudo-divisor as in Proposition \ref{prop:semispec}. Thus $P_{\E,D}(X)\subset P_{\E_1,D_1}(X)$. 

We claim that 
$P_{\E,D}^\circ(X)\subset P_{\E_1,D_1}^\circ(X)$. 
 Let $\D$ be a divisor on $\Gamma$ with combinatorial type $(\E,D)$. Recall that $p_{\D,e}$ is the point in $e^\circ$ such that $\D(p_{\D,e})=-1$. For $e\in E(V,V^c)\setminus \E$, let $p_{\D,e}:=v_e$. Recall that $\D$ corresponds to a point in $K_{\E,D}^\circ(X)$ and hence to a point in $P_{\E,D}^\circ(X)$ (and vice-versa, every point in $P_{\E,D}^\circ(X)$ corresponds to a divisor $\D$ of this form). For each $e\in E(V,V^c)$, let $v_e,w_e\in V(\Gamma)$ be the vertices incident to $e$ with $v_e\in V$ and $w_e\in V^c$. Let $Z$ be the subcurve of $X$:
 \[
 Z=X_V\cup \bigcup_{e\in E(V,V^c)\cap \E} \ol{v_ep_{\D,e}}. 
 \]
 
  As in the proof of Theorem \ref{thm:poly}, item (1), take $r=\min(\ell(\ol{p_{\D,e}w_e}); e\in E(V,V^c))$. Consider the rational function $f$ with slope $0$ everywhere and slope $1$ on $\ora{p_{\D,e}q_e}$, where $q_e$ is the point in $\ol{p_{\D,e}w_e}$ such that $\ell(\ol{p_{\D,e}q_e})=r/2$.  Hence $\D+\Div(f)$ has combinatorial type $(\E_1,D_1)$ (recall Equation \eqref{eq:D1}). This means that the points in $K_{\E_1,D_1}(X)$ corresponding to $\D$ and $\D+\Div(f)$ have the same image in $P_{\E_1,D_1}(X)$ (recall Proposition \ref{prop:parameter}). However, the point associated to $\D+\Div(f)$ lies in the interior $K_{\E_1,D_1}^\circ(X)$, and hence in the interior $P_{\E_1,D_1}^\circ(X)$. Hence we get an inclusion $P_{\E,D}^\circ(X)\subset P_{\E_1,D_1}^\circ(X)
$, proving the claim. \par

 Using Equation \eqref{eq:seqsemi}, we can iterate the argument and obtain the result.
 \end{proof}

\begin{Prop}
\label{prop:face2}
If $(\E,D)$ and $(\E',D')$ are $\mu$-polystable pseudo-divisors on $\Gamma$ such that $(\E',D')\geq (\E,D)$, then $P_{\E,D}(X)$ is a face of $P_{\E',D'}(X)$. Conversely, every face of $P_{\E',D'}(X)$ arises in this way.
\end{Prop}
\begin{proof}
  	 Let $(\E',D')$ be a $\mu$-polystable pseudo-divisor on $\Gamma$. First, 
  	 note that every face of $P_{\E',D'}(X)$ is the image of a face of $K_{\E',D'}(X)$, hence it is of the form $P_{\E_0,D_0}(X)$ for some $\mu$-semistable pseudo-divisor $(\E_0,D_0)$ (recall Remark \ref{rem:semispec}). By Proposition \ref{prop:semipolyint}, 
  	 \[
  	 P_{\E_0,D_0}^{\circ}\subset P_{\pol({\E_0,D_0})}^\circ\subset P_{\E',D'}
  	 \]
  	 hence $P_{\E_0,D_0} = P_{\pol({\E_0,D_0})}$. This proves the second statement. 
  	 
  	 Second, let $(\E,D)$ be a $\mu$-polystable pseudo-divisor on $\Gamma$ with $(\E',D')\geq (\E,D)$.  Hence $P_{\E,D}(X)$ is contained in a minimal face $P_{\E'',D''}(X)$ (see Equation \eqref{eq:PED}), where $(\E'',D'')$ is $\mu$-polystable and $(\E'',D'')\geq (\E',D')$. Thus $P^\circ_{\E',D'}(X)\cap P^\circ_{\E'',D''}(X)\neq\emptyset$. By Proposition \ref{prop:parameter} and Theorem \ref{thm:poly},  we have $(\E'',D'')=(\E',D')$, and we are done.
\end{proof}

\begin{Prop}\label{prop:ranked}
Let $\Gamma$ be a graph and $\mu$ a  polarization on $\Gamma$. The poset $\ps_{\mu}(\Gamma)$ is ranked of dimension $b_1(\Gamma)$ and is connected in codimension $1$.
\end{Prop}

 \begin{proof}
 We begin noting that the maximal elements $(\E,D)\in \ps_\mu(\Gamma)$ are the ones that satisfy $\rk(\E,D)=b_1(\Gamma)$ (recall that $\rk(\E,D)\le b_1(\Gamma)$).\par
   By Lemma \ref{lem:rkstrict}, if $\rk(\E,D)=b_1(\Gamma)$, then $(\E,D)$ is maximal. Conversely, if $(\E,D)$ is a $\mu$-polystable pseudo-divisor, then, by Proposition \ref{prop:polyquasi}, there is a $(v_0,\mu)$-quasistable pseudo-divisor $(\E',D')$ such that $\pol(\E',D')=(\E,D)$. By \cite[Proposition]{AP1} there exists a $(v_0,\mu)$-quasistable pseudo-divisor $(\E_0,D_0)$ with $\rk(\E_0,D_0)=b_1(\Gamma)$ and $(\E_0,D_0)\geq (\E',D')$. However, this means that $\pol(\E_0,D_0)\geq (\E_0,D_0)\geq (\E',D')$, hence $\pol(\E_0,D_0)\geq\pol(\E',D')= (\E,D)$. But $\rk(\pol(\E_0,D_0)\geq \rk(\E_0,D_0)=b_1(\Gamma)$, which means that $\rk(\pol(\E_0,D_0))=b_1(\Gamma)$. Therefore, every $\mu$-polystable divisor $(\E,D)$ is less or equal than a $\mu$-polystable divisor with rank $b_1(\Gamma)$. \par
   Now, every maximal chain in $\ps_\mu(\Gamma)$ ends in a maximal element $(\E,D)$. By Proposition \ref{prop:face2}, the maximal chains ending in $(\E,D)$ correspond precisely the maximal chains of faces of $P_{\E,D}(X)$, which all have length $\dim P_{\E,D}(X)=\rk(\E,D)=b_1(\Gamma)$. This proves that $\ps_{\mu}(\Gamma)$ is ranked of dimension $b_1(\Gamma)$. \par
 
   We now prove that $\ps_{\mu}(\Gamma)$ is connected in codimension $1$. Let $(\E_1,D_1)$ and $(\E_2,D_2)$ be two $\mu$-polystable pseudo-divisors on $\Gamma$ with rank $b_1(\Gamma)$. By Proposition \ref{prop:polyquasi}, we can consider $(v_0,\mu)$-quasistable divisors $(\E_1',D_1')$ and $(\E_2',D_2')$ on $\Gamma$ such that $\pol(\E_i',D_i')=(\E_i,D_i)$ and $\rk(\E_i',D_i')=g$, for $i=1,2$.
   We know by \cite[Proposition 4.13]{AP1} that $\qs_{v_0,\mu}(\Gamma)$ is connected in codimension $1$. Hence there exists a path in codimension $1$ in  $\qs_{v_0,\mu}(\Gamma)$ connecting $(\E_1',D_1')$ with $(\E_2',D_2')$. Applying the map  $\pol$ to the whole sequence, and recalling Equation \eqref{eq:rkm} and the fact that $\pol$ is order-preserving, we get a sequence in codimension $1$ in $\ps_\mu(\Gamma)$ connecting $(\E_1,D_1)$ with $(\E_2,D_2)$.
 \end{proof}
 
 \begin{Thm}\label{thm:rankuniv}
 The poset $\ps_{\mu,g}$ is ranked of dimension $4g-3$ and connected in codimension $1$.
 \end{Thm}
 \begin{proof}
 The proof is essentially the same as in \cite[Theorem 4.15]{AP1}, combining Proposition \ref{prop:ranked} with the same results for the poset of genus-$g$ stable weighted graphs in \cite[Theorem 3.2.5]{BMV} and \cite[Fact 4.12]{Caporaso}.
 \end{proof}

\begin{Def}
Let $X$ be a tropical curve with a polarization $\mu$ and $\mu$-model $\Gamma$.
The \emph{$\mu$-polystable Jacobian} of $X$ is the polyhedral complex
\[
P^{\trop}_\mu(X)=\lim_{\longrightarrow} P_{\E,D}(X)
\]
where the limit is taken over the poset $\PS_\mu(\Gamma)$. In particular, we have
\[
P^{\trop}_\mu(X)=\coprod_{(\E,D)\in \PS_\mu(\Gamma)} P^\circ_{\E,D}(X).
\]
\end{Def}

We have the following theorem giving a polyhedral decomposition of the tropical Jacobian (see also \cite[Proposition 5.8]{CPS}).

\begin{Thm}\label{thm:decomposition}
Given a tropical curve $X$, there is a homeomorphism $P^{\trop}_\mu(X)\to J(X)$.
\end{Thm}
\begin{proof}
Since $P^{\trop}_\mu(X)$ is compact and $J(X)$ is Hausdorff it is sufficient to construct a continuous bijective map $\alpha\col P^{\trop}_\mu(X)\to J(X)$. Fix a point $p_0\in X$, and let $\alpha$ be the map taking a $\mu$-polystable divisor $\D$ on $\Gamma$ to the class of the divisor $\D-dp_0$. It is a bijection by Theorem \ref{thm:poly}. It is continuous, because each map  $\alpha|_{P_{\E,D}(X)}\col P_{\E,D}(X)\to J(X)$ is equal to the continuous map of Equation \eqref{eq:Pcont}.
\end{proof}

Recall that, given a tropical curve $X$ and a point $p_0\in X$, we define the Jacobian of $X$ with respect to $(p_0,\mu)$ as the polyhedral complex:
\begin{equation}\label{eq:quasiJacobian}
J^{\trop}_{p_0,\mu}(X)=\underset{\longrightarrow}{\lim} K_{\E,D}(X),
\end{equation}
where $(\E,D)$ runs through all $(p_0,\mu)$-quasistable divisors (see \cite[Definition 5.7]{AP1}).\par
 
 We have the following proposition, see also \cite[Corollary 5.10]{CPS}.

\begin{Prop}\label{prop:refinement}
Let $X$ be a tropical curve and $p_0$ be a point of $X$. We have a refinement map of polyhedral complexes $J^{\trop}_{p_0,\mu}(X)\to P^{\trop}_\mu(X)$.
\end{Prop}
\begin{proof}
 A $(p_0,\mu)$-quasistable divisor $(\E,D)$ is simple by \cite[Proposition 4.6]{AP1}. Thus Equation \eqref{eq:Ksimple} implies that $K_{\E,D}(X)=P_{\E,D}(X)$. By Proposition \ref{prop:minimal} and Equation \eqref{eq:PED}, we have inclusions $P_{\E,D}(X)\subset P_{\pol(\E,D)}(X)$. Hence we obtain a refinement map $J^{\trop}_{p_0,\mu}(X)\to P^{\trop}_\mu(X)$ of polyhedral complexes.
\end{proof}

\begin{Exa}
Consider the tropical curve of Example \ref{exa:KE} and the polarization $\mu$ of Example \ref{exa:polyposet}. In Figure \ref{fig:polyjac} we draw a picture of the Jacobian $P_\mu^{\trop}(X)$ with its natural  polyhedral decomposition. This is a hexagon (with suitable identifications). One can check that $J^{\trop}_{p_0,\mu}(X)$ is as in  \cite[Figure 4]{AP1}, and it is clear that we have a refinement map $J^{\trop}_{p_0,\mu}(X)\ra P_\mu^{\trop}(X)$. 
\begin{figure}[ht]
\begin{tikzpicture}[scale=2.5]
%\draw[ultra thick] (0,0) rectangle (1,1);
%\draw[ultra thick] (0,0) -- (0,1) -- (-1,0) -- (-1,-1) -- (0,0);
\draw[ultra thick]  (0,1) -- (-1,0) -- (-1,-1) -- (0,-1) -- (1,0) -- (1,1) -- (0,1);
%\draw[ultra thick] (0,0) -- (1,0) -- (0,-1) -- (-1,-1) -- (0,0);
%\draw[fill] (0,0) circle [radius=0.03];
\draw[fill] (0,1) circle [radius=0.03];
\draw[fill] (1,0) circle [radius=0.03];
\draw[fill] (1,1) circle [radius=0.03];
\draw[fill] (0,-1) circle [radius=0.03];
\draw[fill] (-1,0) circle [radius=0.03];
\draw[fill] (-1,-1) circle [radius=0.03];
\begin{scope}[shift={(-0.3,0)},scale=0.8]
\draw (0,0) to [out=87, in=88] (1,0);
%\draw (0,0) to [out=45, in=135] (1,0);
\draw (0,0) to (1,0);
\draw (0,0) to [out=-87, in=-88] (1,0);
\draw[fill] (0,0) circle [radius=0.02];
\draw[fill] (1,0) circle [radius=0.02];
\draw[fill] (0.5,0.29) circle [radius=0.02];
\draw[fill] (0.5,0) circle [radius=0.02];
\draw[fill] (0.5,-0.29) circle [radius=0.02];
\node[left] at (0,0) {1};
\node[right] at (0.95,0) {1};
\node[above] at (0.6,0.23) {-1};
\node[above] at (0.6,-0.062) {-1};
\node[below] at (0.6,-0.07) {-1};
\end{scope}
\begin{scope}[shift={(-0.7,-1.3)},scale=0.5]
\draw (0,0) to [out=87, in=88] (1,0);
%\draw (0,0) to [out=45, in=135] (1,0);
\draw (0,0) to (1,0);
\draw (0,0) to [out=-87, in=-88] (1,0);
\draw[fill] (0,0) circle [radius=0.02];
\draw[fill] (1,0) circle [radius=0.02];
%\draw[fill] (0.5,0.29) circle [radius=0.02];
%\draw[fill] (0.5,0) circle [radius=0.02];
\draw[fill] (0.5,-0.29) circle [radius=0.02];
\node[left] at (0.05,0) {0};
\node[right] at (0.95,0) {0};
%\node[above] at (0.6,0.23) {-1};
%\node[above] at (0.6,-0.06) {-1};
\node[below] at (0.6,0.05) {-1};
\end{scope}
\begin{scope}[shift={(-1.1,0.6)},scale=0.5]
\draw (0,0) to [out=87, in=88] (1,0);
%\draw (0,0) to [out=45, in=135] (1,0);
\draw (0,0) to (1,0);
\draw (0,0) to [out=-87, in=-88] (1,0);
\draw[fill] (0,0) circle [radius=0.02];
\draw[fill] (1,0) circle [radius=0.02];
%\draw[fill] (0.5,0.29) circle [radius=0.02];
\draw[fill] (0.5,0) circle [radius=0.02];
%\draw[fill] (0.5,-0.29) circle [radius=0.02];
\node[left] at (0.05,0) {0};
\node[right] at (0.95,0) {0};
%\node[above] at (0.6,0.23) {-1};
\node[above] at (0.6,-0.06) {-1};
%\node[below] at (0.6,-0.23) {-1};
\end{scope}
\begin{scope}[shift={(1.11,0.59)},scale=0.5]
\draw (0,0) to [out=87, in=88] (1,0);
%\draw (0,0) to [out=45, in=135] (1,0);
\draw (0,0) to (1,0);
\draw (0,0) to [out=-87, in=-88] (1,0);
\draw[fill] (0,0) circle [radius=0.02];
\draw[fill] (1,0) circle [radius=0.02];
\draw[fill] (0.5,0.29) circle [radius=0.02];
%\draw[fill] (0.5,0) circle [radius=0.02];
%\draw[fill] (0.5,-0.29) circle [radius=0.02];
\node[left] at (0.05,0) {0};
\node[right] at (0.95,0) {0};
\node[above] at (0.6,0.23) {-1};
%\node[above] at (0.6,-0.06) {-1};
%\node[below] at (0.6,-0.23) {-1};
\end{scope}
\begin{scope}[shift={(0.25,1.2)},scale=0.5]
\draw (0,0) to [out=87, in=88] (1,0);
%\draw (0,0) to [out=45, in=135] (1,0);
\draw (0,0) to (1,0);
\draw (0,0) to [out=-87, in=-88] (1,0);
\draw[fill] (0,0) circle [radius=0.02];
\draw[fill] (1,0) circle [radius=0.02];
%\draw[fill] (0.5,0.29) circle [radius=0.02];
%\draw[fill] (0.5,0) circle [radius=0.02];
\draw[fill] (0.5,-0.29) circle [radius=0.02];
\node[left] at (0,0) {0};
\node[right] at (0.95,0) {0};
%\node[above] at (0.6,0.23) {-1};
%\node[above] at (0.6,-0.06) {-1};
\node[below] at (0.6,0.05) {-1};
\end{scope}
\begin{scope}[shift={(-1.64,-0.5)},scale=0.5]
\draw (0,0) to [out=87, in=88] (1,0);
%\draw (0,0) to [out=45, in=135] (1,0);
\draw (0,0) to (1,0);
\draw (0,0) to [out=-87, in=-88] (1,0);
\draw[fill] (0,0) circle [radius=0.02];
\draw[fill] (1,0) circle [radius=0.02];
\draw[fill] (0.5,0.29) circle [radius=0.02];
%\draw[fill] (0.5,0) circle [radius=0.02];
%\draw[fill] (0.5,-0.29) circle [radius=0.02];
\node[left] at (0,0) {0};
\node[right] at (0.95,0) {0};
\node[above] at (0.6,0.23) {-1};
%\node[above] at (0.6,-0.06) {-1};
%\node[below] at (0.6,-0.23) {-1};
\end{scope}
\begin{scope}[shift={(0.64,-0.6)},scale=0.5]
\draw (0,0) to [out=87, in=88] (1,0);
%\draw (0,0) to [out=45, in=135] (1,0);
\draw (0,0) to (1,0);
\draw (0,0) to [out=-87, in=-88] (1,0);
\draw[fill] (0,0) circle [radius=0.02];
\draw[fill] (1,0) circle [radius=0.02];
%\draw[fill] (0.5,0.29) circle [radius=0.02];
\draw[fill] (0.5,0) circle [radius=0.02];
%\draw[fill] (0.5,-0.29) circle [radius=0.02];
\node[left] at (0,0) {0};
\node[right] at (0.95,0) {0};
%\node[above] at (0.6,0.23) {-1};
\node[above] at (0.6,-0.06) {-1};
%\node[below] at (0.6,-0.23) {-1};
\end{scope}
\begin{scope}[shift={(-1.5,0)},scale=0.3]
\draw (0,0) to [out=87, in=88] (1,0);
%\draw (0,0) to [out=45, in=135] (1,0);
\draw (0,0) to (1,0);
\draw (0,0) to [out=-87, in=-88] (1,0);
\draw[fill] (0,0) circle [radius=0.02];
\draw[fill] (1,0) circle [radius=0.02];
%\draw[fill] (0.5,0.29) circle [radius=0.02];
%\draw[fill] (0.5,0) circle [radius=0.02];
%\draw[fill] (0.5,-0.29) circle [radius=0.02];
\node[left] at (0,0) {0};
\node[right] at (0.95,0) {-1};
%\node[above] at (0.6,0.23) {-1};
%\node[above] at (0.6,-0.06) {-1};
%\node[below] at (0.6,-0.23) {-1};
\end{scope}
\begin{scope}[shift={(1.2,0)},scale=0.3]
\draw (0,0) to [out=87, in=88] (1,0);
%\draw (0,0) to [out=45, in=135] (1,0);
\draw (0,0) to (1,0);
\draw (0,0) to [out=-87, in=-88] (1,0);
\draw[fill] (0,0) circle [radius=0.02];
\draw[fill] (1,0) circle [radius=0.02];
%\draw[fill] (0.5,0.29) circle [radius=0.02];
%\draw[fill] (0.5,0) circle [radius=0.02];
%\draw[fill] (0.5,-0.29) circle [radius=0.02];
\node[left] at (0.1,0) {-1};
\node[right] at (0.95,0) {0};
%\node[above] at (0.6,0.23) {-1};
%\node[above] at (0.6,-0.06) {-1};
%\node[below] at (0.6,-0.23) {-1};
\end{scope}
\begin{scope}[shift={(0.1,-1.1)},scale=0.3]
\draw (0,0) to [out=87, in=88] (1,0);
%\draw (0,0) to [out=45, in=135] (1,0);
\draw (0,0) to (1,0);
\draw (0,0) to [out=-87, in=-88] (1,0);
\draw[fill] (0,0) circle [radius=0.02];
\draw[fill] (1,0) circle [radius=0.02];
%\draw[fill] (0.5,0.29) circle [radius=0.02];
%\draw[fill] (0.5,0) circle [radius=0.02];
%\draw[fill] (0.5,-0.29) circle [radius=0.02];
\node[left] at (0.05,0) {0};
\node[right] at (0.95,0) {-1};
%\node[above] at (0.6,0.23) {-1};
%\node[above] at (0.6,-0.06) {-1};
%\node[below] at (0.6,-0.23) {-1};
\end{scope}
\begin{scope}[shift={(-1.4,-1.1)},scale=0.3]
\draw (0,0) to [out=87, in=88] (1,0);
%\draw (0,0) to [out=45, in=135] (1,0);
\draw (0,0) to (1,0);
\draw (0,0) to [out=-87, in=-88] (1,0);
\draw[fill] (0,0) circle [radius=0.02];
\draw[fill] (1,0) circle [radius=0.02];
%\draw[fill] (0.5,0.29) circle [radius=0.02];
%\draw[fill] (0.5,0) circle [radius=0.02];
%\draw[fill] (0.5,-0.29) circle [radius=0.02];
\node[left] at (0,0) {-1};
\node[right] at (0.95,0) {0};
%\node[above] at (0.6,0.23) {-1};
%\node[above] at (0.6,-0.06) {-1};
%\node[below] at (0.6,-0.23) {-1};
\end{scope}
\begin{scope}[shift={(-0.4,1.1)},scale=0.3]
\draw (0,0) to [out=87, in=88] (1,0);
%\draw (0,0) to [out=45, in=135] (1,0);
\draw (0,0) to (1,0);
\draw (0,0) to [out=-87, in=-88] (1,0);
\draw[fill] (0,0) circle [radius=0.02];
\draw[fill] (1,0) circle [radius=0.02];
%\draw[fill] (0.5,0.29) circle [radius=0.02];
%\draw[fill] (0.5,0) circle [radius=0.02];
%\draw[fill] (0.5,-0.29) circle [radius=0.02];
\node[left] at (0,0) {-1};
\node[right] at (0.94,0) {0};
%\node[above] at (0.6,0.23) {-1};
%\node[above] at (0.6,-0.06) {-1};
%\node[below] at (0.6,-0.23) {-1};
\end{scope}
\begin{scope}[shift={(1,1.1)},scale=0.3]
\draw (0,0) to [out=87, in=88] (1,0);
%\draw (0,0) to [out=45, in=135] (1,0);
\draw (0,0) to (1,0);
\draw (0,0) to [out=-87, in=-88] (1,0);
\draw[fill] (0,0) circle [radius=0.02];
\draw[fill] (1,0) circle [radius=0.02];
%\draw[fill] (0.5,0.29) circle [radius=0.02];
%\draw[fill] (0.5,0) circle [radius=0.02];
%\draw[fill] (0.5,-0.29) circle [radius=0.02];
\node[left] at (0,0) {0};
\node[right] at (0.95,0) {-1};
%\node[above] at (0.6,0.23) {-1};
%\node[above] at (0.6,-0.06) {-1};
%\node[below] at (0.6,-0.23) {-1};
\end{scope}
\end{tikzpicture}
\caption{The Jacobian $P_{\mu}^{\trop}(X)$.}
\label{fig:polyjac}
\end{figure}
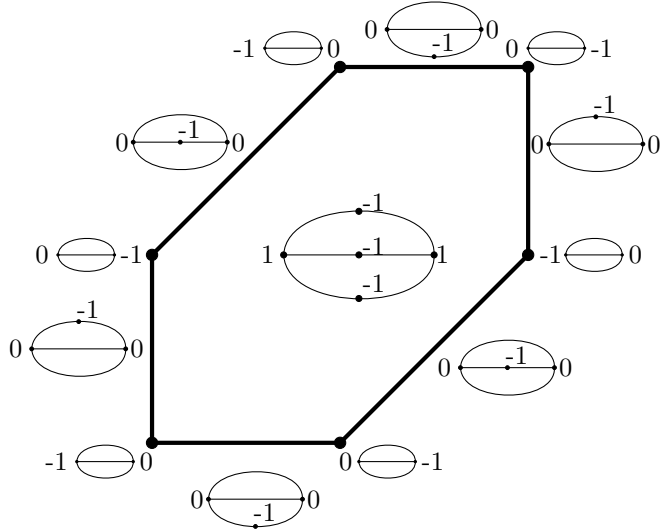\end{Exa}

\subsection{A universal tropical Jacobian} 

We are in a position to introduce a universal tropical Jacobian over $M_g^{\trop}$. The idea is to glue together all the cones $\sigma_{(\Gamma,\E,D)}$, for all graphs $\Gamma$ of genus $g$ and all $\mu$-polystable pseudo-divisors $(\E,D)$ on $\Gamma$. First of all, we have the following analogue of Propositions \ref{prop:semipolyint} and \ref{prop:face2} for the cones $\sigma_{(\Gamma,\E,D)}$. 

\begin{Prop}
\label{prop:face3}
Let $\Gamma$ be a graph and $(\E,D)$ be a $\mu$-semistable pseudo-divisor on $\Gamma$. The following properties hold.
\begin{enumerate}
    \item If we let $(\E',D')=\pol(\E,D)$, then 
$\sigma_{(\Gamma,\E,D)}^\circ\subset \sigma_{(\Gamma,\E',D')}^\circ$.
\item 
Given a $\mu$-polystable pseudo-divisor $(\E',D')$ such that $(\E',D')\geq(\E,D)$, we have that $\sigma_{(\Gamma,\E,D)}$ is a face of $\sigma_{(\Gamma,\E',D')}$ if and only if $(\E,D)$ is $\mu$-polystable.
\end{enumerate}

\end{Prop}
\begin{proof}
Just use the results in Propositions \ref{prop:semipolyint} and \ref{prop:face2} together with the fact that $\pi_\E^{-1}(X)\cong P_{\E,D}(X)$ from Proposition \ref{prop:parameter1}.
\end{proof}

\begin{Prop}\label{prop:glue}
Let $(\E_1,D_1)$ and $(\E_2,D_2)$ be $\mu$-polystable divisors on the graphs $\Gamma_1$ and $\Gamma_2$.
If $(\Gamma_1,\E_1,D_1)\ge (\Gamma_2,\E_2,D_2)$, then there exists a natural face morphism $\sigma_{(\Gamma_2,\E_2,D_2)}\to\sigma_{(\Gamma_1,\E_1,D_1)}$. Conversely, every face of $\sigma_{(\Gamma_1,\E_1,D_1)}$ arises in this way.
\end{Prop}
\begin{proof}
Assume that $(\Gamma_1,\E_1,D_1)\ge (\Gamma_2,\E_2,D_2)$. Let $\iota\col \Gamma_1\to \Gamma_2$ be the induced specialization and set $(\E',D')=\iota_*(\E_1,D_1)$. We have that 
\[
(\Gamma_1,\E_1,D_1)\ge (\Gamma_2,\E',D')\ge (\Gamma_2,\E_2,D_2).
\] 
 By Proposition \ref{prop:face1} the cone $\sigma_{(\Gamma_2,\E',D')}$ is a face of $\sigma_{(\Gamma_1,\E_1,D_1)}$. 
 By Lemma \ref{lem:polyspec} we have that $(\E',D')$ is $\mu$-polystable. Then, by Lemma \ref{prop:face3}, the cone $\sigma_{(\Gamma_2,\E_2,D_2)}$ is a face of $\sigma_{(\Gamma_2,\E',D')}$, and consequently a face of $\sigma_{(\Gamma_1,\E_1,D_1)}$.

 Conversely, the faces of $\sigma_{(\Gamma_1,\E_1,D_1)}$ must be images of faces of $\R^{E(\Gamma^\E)}_{\geq0}$ and hence, by Remark \ref{rem:semispec}, they are of the form $\sigma_{(\Gamma_2,\E_2,D_2)}$, where $(\E_2,D_2)$ is a $\mu$-semistable pesudo-divisor on $\Gamma_2$ such that $(\Gamma_1,\E_1,D_1)\ge (\Gamma_2,\E_2,D_2)$. As before, denote by $(\E',D'):=\iota_*(\E_1,D_1)$, where $\iota\col \Gamma_1\to \Gamma_2$ is the induced specialization. Since $(\E',D')\geq(\E_2,D_2)$, we have $(\E',D')\geq \pol(\E_2,D_2)$ by Proposition \ref{prop:minimal}. Hence $\sigma_{\Gamma_2,\pol(\E_2,D_2)}$ is a face $\sigma_{(\Gamma_1,\E_1,D_1)}$ by the first part of the proof.  It follows from Proposition \ref{prop:face3} that $\sigma_{(\Gamma_2,\E_2,D_2)}^\circ\subset \sigma_{(\Gamma_2,\pol(\E_2,D_2))}^\circ$, which implies that $\sigma_{(\Gamma_2,\E_2,D_2)}=\sigma_{(\Gamma_2,\pol(\E_2,D_2))}$. Hence, every face of $\sigma_{(\Gamma_1,\E_1,D_1)}$ comes from a $\mu$-polystable pseudo-divisor.
\end{proof}

\begin{Exa}\label{exa:Karl}
There is a subtle aspect about
Proposition \ref{prop:glue} that Karl Christ pointed out to us and that we want to illustrate in this example. Consider the tropical curve $X$ and its model $\Gamma$ as in Figure \ref{fig:Karl}. Let $(\E,D)$ be a pseudo-divisor on $\Gamma$, where $\E=E(\Gamma)$. The cone $\sigma_{(\Gamma,\E,D)}$ is equal to $\R^4_{\geq0}/L$, where $L$ is the linear subspace generated by the vector $(1,1,-1,-1)$ (the coordinates are $(x,y,z,w)$). Identifying $\R^4/L$ with $\R^3$, where $\ol{e_4}=\ol{e_1}+\ol{e_2}-\ol{e_3}$, we can think of $\sigma_{(\Gamma,\E,D)}$ as the cone generated by the vectors $(1,0,0)$, $(0,1,0)$, $(0,0,1)$ and $(1,1,-1)$.
There are 4 possible specializations: 
\[
x=z=0, \;\; y=w=0,\;\; x=w=0, \;\; z=y=0,
\]
which correspond to the faces 
\[
\langle (0,1,0),(1,1,-1)\rangle, \;\; \langle (1,0,0),(0,0,1)\rangle, \;\; \langle (0,1,0),(0,0,1)\rangle, \;\;
\langle (1,0,0),(1,1,-1)\rangle.
\]
 Note that we cannot consider the specializations $x=0$, $y=0$, $z=0$, $w=0$, $x=y=0$, $z=w=0$, as the combinatorial type we get after contraction is not polystable.
The potential issue is that   two equivalent polystable divisors might not remain equivalent upon contraction.
%Let us try to explain that it is not an issue (and hence the faces above glue nicely).
%, I believe that the reasoning below is a suitable explanation, and perhaps the reason why .
This, though, does not happen. In fact, fix a parameter $t>0$ and consider a family of tropical curves $\Gamma(t)$ as in Figure \ref{fig:Karl}, with lengths $\ell_1(t)>0$ and $\ell_2>0$, and assume that $\lim_{t\to0}\ell_1(t)=0$. Consider, for each $t>0$, two equivalent polystable divisors $D_1(t)$ and $D_2(t)$ with  combinatorial type as in Figure \ref{fig:Karl}. Assume that $D_i(t)$ is given by the pair $(x_i(t),y_i(t))$, where $x_i(t), y_i(t)$ are the distances to the vertex $v_0$. These two divisor are equivalent, if, and only if, $x_1(t)-x_2(t)=y_1(t)-y_2(t)$.
 However, $|x_1(t)-x_2(t)|<\ell_1(t)$, hence $\lim_{t\to 0} (y_1(t)-y_2(t))=0$. We deduce that, if the limits $\lim_{t\to 0}D_1(t)$ and $\lim_{t\to 0} D_2(t)$ exist, then they must be the same.
   In a way, when an edge is contracted, one has to do it ``continuously": the points on the other edge are forced to come together.
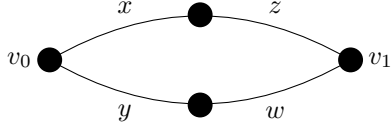
\begin{figure}[h]
    \centering
    \begin{tikzpicture}
    \draw (0,0) to [bend right] (4,0);
    \draw (0,0) to [bend left] (4,0);
    \node at (0,0) [circle, fill]{};
    \node at (-0.4,0) {$v_0$};
    \node at (4.4,0) {$v_1$};
    \node at (4,0) [circle, fill]{};
    \node at (2,0.6) [circle, fill]{};
    \node at (2,-0.6) [circle, fill]{};
    \node at (1,-0.7) {$y$};
    \node at (1,0.7) {$x$};
    \node at (3,-0.7) {$w$};
    \node at (3,0.7) {$z$};
    \end{tikzpicture}
\caption{A tropical curve}
\label{fig:Karl}
\end{figure}
\end{Exa}

\begin{Def}
Let $\mu$ be a genus-$g$ universal polarization of degree $d$. 
The \emph{universal tropical Jacobian} $P^{\trop}_{\mu,g}$ over $M^\trop_{g}$ is the generalized cone complex:
\[
P^{\trop}_{\mu,g}=\lim_{\longrightarrow} \sigma_{(\Gamma,\E,D)}=\coprod_{[(\Gamma,\E,D)]}\sigma_{(\Gamma,\E,D)}^\circ/\Aut(\Gamma,\E,D)
\]
where the limit is taken over all triples $(\Gamma,\E,D)$ running through all objects in the category $\PS_{g,\mu}$ and the union is taken over all equivalence classes $[(\Gamma,\E,D)]$ in $\ps_{g,\mu}$.
 If $\mu$ is the canonical genus-$g$ universal polarization of degree $d$, we simply write:
\[
P^{\trop}_{d,g}:=P^{\trop}_{\mu,g}.
\]
\end{Def}

\begin{Rem}
Recall that, by \cite[Section 5]{KP}, there exists a unique genus-$g$ universal polarization of degree-$d$, which is the canonical polarization. Maintaing the label $\mu$, however, enables to generalize in an easy way to other moduli spaces, which do admit other universal polarizations.
\end{Rem}

\begin{Prop}\label{prop:parameteruni}
The generalized cone complex $P^\trop_{\mu,g}$ parametrizes equivalence classes $(X,\D)$, where $X$ is a stable tropical curve of genus $g$ and $\D$ is a $\mu$-polystable divisor on $X$, and two pairs $(X_1,\D_1)$ and $(X_1,\D_2)$ are equivalent if there exists an isomorphism $\iota\col X_1\to X_2$ such that $\iota_*(\D_1)$ is linearly equivalent to $\D_2$.
\end{Prop}

\begin{proof}
By Proposition \ref{prop:parameter1}, each open cone $\sigma_{(\Gamma,\E,D)}^\circ$ parametrizes equivalence classes of pairs $(X,\D)$, where $X$ is a tropical curve with model $\Gamma$ and $\D$ is a unitary divisor on $X$ with combinatorial type $(\E,D)$. Two pairs $(X_1,\D_1)$ and $(X_2,\D_2)$ are equivalent if $X_1=X_2$ and $\D_1$ and $\D_2$ are linearly equivalent.

Let $X$ be a stable tropical curve of genus $g$. The stable model 
$\Gamma$ of $X$ is a genus-$g$ stable weighted graph. %The polarization $\mu_{\Gamma}$ induces a polarization on $X$, such that $\Gamma_X=\Gamma$ (recall that $\Gamma_X$ denotes the model of $X$ whose vertices are the relevant points).\par
   By definition, if $\D$ is a $\mu$-polystable divisor of $X$, then $\D$ is unitary and has combinatorial type  
   equal to some $\mu$-polystable pseudo-divisor $(\E,D)$ on $\Gamma$.  Therefore $(X,\D)$ corresponds to a point in $\R^{E(\Gamma^\E)}_{>0}$, and hence to a point in $\sigma_{(\Gamma,\E,D)}$. Then the pair $(X,\D)$ corresponds to a point in $P^\trop_{\mu,g}$. If $(X,\D)$ and $(X',\D')$ are in the same equivalence class, then there is an isomorphism $\iota\col X\to X'$ such that $\iota_*(D)$ is linearly equivalent to $D'$. Consider the three points $p_1,p_2,p_3$ in $\R^{E(\Gamma^\E)}_{>0}$ corresponding to $(X,\D), (X',\iota_*(\D)),(X',\D')$ . The points $p_1$ and $p_2$ will get identified in  $\R^{E(\Gamma^\E)}_{>0}/\Aut(\Gamma,\E,D)$, while $p_2$ and $p_3$ will identify in $\R^{E(\Gamma^\E)}_{>0}/\EL_\E$. Hence, $p_1$ and $p_3$ will correspond to the same point in 
   \[
   \sigma_{\Gamma,\E,D}^\circ/\Aut(\Gamma,\E,D)=(\R^{E(\Gamma^\E)}_{>0}/\EL_\E)/\Aut(\Gamma,\E,D)\subset P_{\mu,g}^{\trop}.
   \]
   
On the other hand if $(X,\D)$ and $(X',\D')$ corresponds to the same point in $P^\trop_{\mu,g}$ contained in some cell $\sigma^\circ_{(\Gamma,\E,D)}/\Aut(\Gamma,\E,D)$, then there is an isomorphism $\iota\col\Gamma_X\ra\Gamma_{X'}$ such that $\iota_*(\E,D)=(\E',D')$. Moreover, it follows from  Lemma \ref{lem:VEker} that the metrics of $X$ and $X'$ are equal, hence $\iota$ induces an isomorphism of metric graphs $\iota\col X\to X'$. This means that $(X',\D')$ and $(X',\iota_*(\D))$ are the same point in $\sigma_{\Gamma,\E,D}^\circ$. By Proposition \ref{prop:parameter1}, the divisors $\D'$ and $\iota_*(\D)$ are linearly equivalent. Hence $(X,\D)$ and $(X',\D')$ are equivalent, as required.\par

Conversely, given a triple $(\Gamma,\E,D)$, every point in $\R^{E(\Gamma^\E)}_{>0}$ corresponds to a pair $(X,\D)$. By a similar argument as above, we see that if two such points are identified in  $\sigma^\circ_{(\Gamma,\E,D)}/\Aut(\Gamma,\E,D)$, then the pairs are in the same equivalence class.
\end{proof}

\begin{Thm}\label{thm:universalPic}
The generalized cone complex $P^\trop_{\mu,g}$ has dimension $4g-3$ and it is connected in codimension $1$. The natural forgetful map $\pi^{\trop}\col P^\trop_{\mu,g}\to M_g^\trop$ is a map of generalized cone complexes, and we have: 
\[
(\pi^{\trop})^{-1}[X]\cong J(X)/\Aut(X),
\]
for every stable weighted tropical curve $X$ of genus $g$.
\end{Thm}

\begin{proof}
The fact that $P^\trop_{\mu,g}$ has pure dimension $4g-3$ and is connected in codimension $1$ follows from Theorem \ref{thm:rankuniv}.\par
 For each cone $\sigma_{(\Gamma,\E,D)}$, the map $\pi^\trop$ induces a map $\sigma_{(\Gamma,\E,D)}\to M_{g}^\trop$ that factors through a chain of maps
\[
\sigma_{(\Gamma,\E,D)}=T_\E(\R_{\geq0}^{E(\Gamma^\E)}) \to \R_{\geq 0}^{E(\Gamma)}\ra  M_{g}^\trop,
\]
 where the first one is the map defined in Lemma \ref{lem:VEker} composed with the projection on the first factor, and the second one is the natural map. Hence $\pi^{trop}$ is a morphism of generalized cone complexes.\par
  
  For every stable weighted tropical curve $X$ of genus $g$, there is a natural map $h\col P^{\trop}_\mu(X)\to P^\trop_{\mu,g}$ and we have $(\pi^{trop})^{-1}([X])=\Ima(h)$. Moreover, $h(\D)=h(\D')$ if and only if there exists an automorphism $\alpha\col X\to X$ such that $\alpha_*(\D)=\D'$, which implies that $\Ima(h)\cong J(X)/\Aut(X)$.
\end{proof}

%\section{Final comments and further problems}

\subsection{Final comments}\label{sec:strat}

%Let  $X$ and  $Y$ be nodal curves with dual graphs $\Gamma_X$ and $\Gamma_Y$ and let $I_X$ and $I_Y$ be torsion free rank-$1$ sheaves over $X$ and $Y$ with multidegrees $(\E_X,D_X)$. If $(Y,L_Y)$ specializes to $(X,L_X)$, then there exists a specialization  $\iota\col (\Gamma_X,\E_X,D_X) \to (\Gamma_Y,\E_Y,D_Y)$. On the other hand, if $\iota\col (\Gamma_X,\E_X,D_X)\to (\Gamma,\E,D)$ is a specializations, there exists a nodal curve $Z$ with dual graph $\Gamma$ and a torsion free renk-$1$ sheaf $L_Z$ on $Z$, with multidegree $(\E,D)$ such that $(Z,L_Z)$ specializes to $(X,L_X)$.\par 

To wrap up this paper, let us make a few observations on  stratifications of some universal compactified Jacobians over the moduli space of stable curves. They all are essentially consequences of \cite[Propositions 3.4.1, 3.4.2]{CC}.

Let $I$ be a torsion-free rank-1 sheaf on a nodal curve $X$. The \emph{combinatorial type} of the pair $(X,I)$ is the triple $(\Gamma,\E,D)$, where $\Gamma$ is the usual dual graph of the nodal curve $X$, the set $\E$ is the subset $\E\subset E(\Gamma)$ of the edges corresponding to the nodes over which $I$ fails to be invertible, and $D$ is the divisor on $\Gamma^\E$ such that $D(v)=-1$ if $v\in V(\Gamma^\E)$ is exceptional, while $D(v)=\deg(I|_{X_v})$ if $v\in V(\Gamma)$ is not exceptional (here $X_v$ is the component of $X$ corresponding to $v$).

As we saw in the introduction, the compactified Picard scheme $\ol{P}_{d,g}$ over  $\ol{M}_g$ parametrizes isomorphism classes of stably balanced line bundles on quasistable curves or, equivalently, pairs $(X,I)$, where $X$ is a stable curve of genus $g$ and $I$ is a $\mu$-polystable torsion-free rank-$1$ sheaf on $X$ of degree $d$ (see \cite{C} and \cite{Pand}). For every graph $\Gamma$ and every pseudo-divisor $(\E,D)$ on $\Gamma$, we let $P_{\Gamma,\E,D}\subset \ol{P}_{d,g}$ be the subscheme where $(X,I)$ has combinatorial type isomorphic to $(\Gamma,\E,D)$. Then we have a stratification: 
\[
\ol{P}_{d,g}=\bigsqcup_{(\Gamma,\E,D)\in \ps_{d,g}}P_{\Gamma,\E,D}.
\]

There also is the Jacobian $\ol{\mathcal{J}}_{d,g}$ over $\ol{\mathcal M}_g$ introduced in \cite{AK}, \cite{Es01} and \cite{M15}. This is the Deligne-Mumford stack parametrizing isomorphism classes of pairs $(X,I)$ where $X$ is a stable curve of genus $g$ and $I$ is a simple torsion-free rank-$1$ sheaf of degree $d$ on $X$. We have a stratification:
\[
\ol{\J}_{d,g}=\bigsqcup_{(\Gamma,\E,D)}\J_{\Gamma,\E,D},
\]
where $(\Gamma,\E,D)$ runs through all stable weighted graphs $\Gamma$ of genus $g$ and simple pseudo-divisors $(\E,D)$ on $\Gamma$. 
Recall that $\ol{\mathcal{J}}_{d,g}$ is neither separated nor of finite type over  $\ol{\mathcal{M}}_g$. 

Finally, we have the compactified Jacobian $\ol{\mathcal{J}}^{ss}_{\mu,g}$. This is the Deligne-Mumford stack over $\ol{\mathcal M}_g$ parametrizing pairs $(X,I)$ where $X$ is a stable curve of genus $g$ and $I$ is a $\mu$-semistable simple torsion-free rank-$1$ sheaf on $X$ (see \cite{Es01} and \cite{M15}). We have a   stratification:
\[
\ol{\J}^{ss}_{\mu,g}=\bigsqcup_{(\Gamma,\E,D)} \J_{\Gamma,\E,D},
\]
where $(\Gamma,\E,D)$ runs through  stable weighted graphs $\Gamma$ of genus $g$ and simple $\mu$-semistable pseudo-divisors $(\E,D)$ on $\Gamma$. Recall that $\ol{\J}^{ss}_{\mu,g}$ is not separated over  $\ol{\mathcal{M}}_g$. 
In the above formulas, $\J_{\Gamma,\E,D}$ is the locus parametrizing pairs $(X,I)$ whose combinatorial type is isomorphic to $(\Gamma,\E,D)$.

%This paper raises the following  natural question:  is there a Deligne-Mumford toroidal universal compactified Jacobian whose skeleton (in the sense of \cite{ACP}) is equal to (the compactification of) $P_{\mu,g}^{\trop}$? The natural candidates should be Caporaso-Pandharipande universal  scheme, or its associated stack of all semistable sheaves on quasistable curves. 
%It is not clear to us that the former is toroidal, and the latter is not Deligne-Mumford in the degenerate case.   

%\begin{enumerate}
 %   \item Construir $P_{\Gamma,\E,D}$ estrato em $P_{d,g}$.
  %  \item $P_{\Gamma,\E,D} = J_{\Gamma_\E,D_\E+F}/\Aut(\Gamma,\E,D)$ ou $P_{\Gamma_\E,D_\E,+F}\to P_{\Gamma,\E,D}$
   % \item 
%\end{enumerate}

%\section{Final remarks and further problems}

%\subsection{Semistable and polystable}

 %  \begin{enumerate}
  %     \item Limit of simple semistable divisors
   %    \item Limit of all semistable divisors
  % \end{enumerate}

%\begin{Question}
%Does every polytope appears in $J^{ps}(\Gamma)$ for some $\Gamma$?
%\end{Question}

\bigskip
\medskip

\noindent{\small Alex Abreu, Sally Andria, Marco Pacini,  and Danny Taboada \\
Universidade Federal Fluminense, Rua Prof. M. W. de Freitas, S/N\\ 
Niter\'oi, Rio de Janeiro, Brazil. 24210-201.}


\begin{thebibliography}{CCC}
\bibitem{AP1} A. Abreu, M. Pacini. \emph{The universal tropical Jacobian and the skeleton of the Esteves'
universal Jacobian.} Proc. London Math. Soc. vol. 120, issue 3, (2020), 328--369.

\bibitem{AP2} A. Abreu, M. Pacini,
\emph{The resolution of the universal Abel map via tropical geometry and applications.}
arXiv:1903.08569. 

\bibitem{AK} A. Altman and S. Kleiman, 
\emph{Compactifying the Picard scheme.}
Adv. Math. {\bf 35} (1980) 50--112.  

\bibitem{ABKS} Y. An, M. Baker, G. Kuperberg, F. Shokrieh,
\emph{Canonical representatives for divisor classes on tropical curves and the matrix-tree theorem.}
For. of Math., Sigma, {\bf 2} (2014) 2--24. 

\bibitem{ACP}  D. Abramovich, L. Caporaso, and S. Payne,
\emph{The tropicalization of the moduli space of curves.}
Annales Scientifiques del'ENS. {\bf 48 (4)} (2015), no. 4, 765--809.

\bibitem{BF} M. Baker and X. Faber,
\emph{Metric properties of the tropical Abel-Jacobi map.}
Journal of Algebraic Combinatorics. {\bf 33} (2011), no. 3, 349-381. 

\bibitem{BR} M. Baker and J. Rabinoff,
\emph{The skeleton of the Jacobian, the Jacobian of the skeleton, and lifting meromorphic functions from tropical to algebraic curves.}
Int. Math. Res. Not., {\bf 16} (2015) 7436--7472.

%\bibitem{Be} V. G. Berkovich,
%\emph{ Spectral theory and analytic geometry over non-Archimedean fields.} 
%Mathematical Surveys and monographs vol. {\bf 33}, American Mathematical
%Society, 1990.


\bibitem{BMV} S. Brannetti, M. Melo, and F. Viviani,
\emph{On the tropical Torelli map.}
Adv. Math. {\bf 226 (3)} (2011) 2546--2586.

\bibitem{C} L. Caporaso,
\emph{A compactification of the universal Picard variety over the moduli space of stable curves.}
Journal of American Mathematics Society, {\bf 7}, N. 3 (1994), 589--660.

%\bibitem{C08} L. Caporaso,
%\emph{N\'eron models and compactified Picard schemes over the moduli stack of stable curves.}
%American Journal of Mathematics, {\bf 130} (2008), 1--47.

\bibitem{Caporaso1} L. Caporaso,
\emph{Geometry of tropical moduli spaces and linkage of graphs.}
 Journal of Combinatorial Theory, Series A. 119 (2012) 579--598.

\bibitem{Caporaso} L. Caporaso, 
\emph{Algebraic and tropical curves: comparing their moduli spaces.} In: Handbook of Moduli, Volume I. G. Farkas, I. Morrison (Eds.), Advanced Lectures in Mathematics, Volume XXIV (2012), 119--160.

\bibitem{CC} L. Caporaso and K. Christ,
\emph{Combinatorics of compactified universal Jacobians.}
Adv. Math. {\bf 346} (2019) 1091--1136.
	
\bibitem{CMP} L. Caporaso, M. Melo, and M. Pacini,
\emph{The tropicalization of the moduli space of spin curves.}
arXiv:1902.07803.


%\bibitem{CKV} S. Casalaina-Martin, J. L. Kass, and F. Viviani,
%\emph{The singularities and birational geometry of the universal compactified Jacobian.}
% Algebraic Geometry {\bf 4} (3) (2017), 353--393. 

%\bibitem{CHMR} R. Cavalieri, S. Hampe, H. Markwig D. Ranganathan,
%\emph{ Moduli spaces of rational weighted stable curves and tropical geometry.}
%Forum of Mathematics {\bf 4} (2016), 1--35.

\bibitem{CMR} R. Cavalieri, H. Markwig, and D. Ranganathan,
\emph{Tropicalizing the space of admissible covers.} 
Math. Ann., {\bf 364} (3-4) (2016) 1275--1313.


\bibitem{CPS} K. Christ, S. Payne, J. Shen, \emph{Compactified Jacobians as Mumford models}. arXiv:1912.03653.

%\bibitem{Ds} C. D'Souza,
%\emph{Compactification of generalized Jacobians.} 
%Proc. Indian Acad. Sci. Sect. A Math. Sci. {\bf 88} (1979), 419--457.

\bibitem{Es01} E. Esteves,
\emph{Compactifying the relative Jacobian over families of reduced curves.}
Trans. Amer. Math. Soc., {\bf 353} (2001), 3345--3095.

%\bibitem{EP} E. Esteves and M. Pacini, 
%\emph{Semistable modifications of families of curves and compactified Jacobians.}
%Ark. Mat. {\bf 54} no. 1, (2016), 55--83.


\bibitem{H} D. Holmes,
\emph{Extending the double ramification cycle by resolving the Abel-Jacobi map}. J. Inst. Math. Jussieu (2019), 1--29.

%\bibitem{I56} J. Igusa,
%\emph{Fibre systems of Jacobian varieties.} 
%Amer. J. Math. {\bf 78} (1956), 171--199.

\bibitem{KP} J. Kass and N. Pagani,
\emph{The stability space of compactified universal Jacobians.}
To appear in Trans. Amer. Math. Soc., arXiv:1707.02284.

\bibitem{Len} Y. Len, \emph{The Brill--Noether rank of a tropical curve} Journal of Algebraic Combinatorics {\bf 40} (2014), 841–-860.

%\bibitem{MM} A. Mayer and D. Mumford,
%\emph{Further comments on boundary points.}
% Unpublished lecture notes distributed at the Amer. Math. Soc. Summer Institute, Woods Hole, 1964.

%\bibitem{M11} M. Melo,
%\emph{Compactified Picard stacks over the moduli stack of stable %curves with marked points.}
%Advances in Mathematics {\bf 226} (2011) No. 1, 727--763.  

\bibitem{M15} M. Melo,
\emph{Compactifications of the universal Jacobian over curves with marked points.}
arXiv:1509.06177.

%\bibitem{MV} M. Melo and F. Viviani,
%\emph{Fine compactified Jacobians.} 
%Math. Nachr. {\bf 285} (2012), no. 8-9, 997--1031. 

%\bibitem{MV14} M. Melo and F. Viviani,
%\emph{The Picard group of the compactified universal Jacobian.}
% Documenta Math. {\bf 19} (2014), 457--507.

\bibitem{M} G. Mikhalkin,
\emph{Moduli spaces of rational tropical curves.} 
 Proc. of G\"okova Geometry-Topology Conference 2006, G\"okova Geometry/Topology Conf., G\"okova, 2007, 39--51.

\bibitem{MZ} G. Mikhalkin and I. Zharkov,
\emph{Tropical curves, their Jacobians and Theta functions.}
 Proceedings of the International Conference on Curves and Abelian Varieties in Honor of Roy Smith's 65th Birthday, in: Contemp. Math., { \bf 465}, 2007, 203--231.


\bibitem{MUW} M. M\"oeller, M. Ulirsch, and A. Werner,
 \emph{Realizability of tropical canonical divisors}.
To appear in J. Eur. Math. Soc., arxiv.org/abs/1710.06401.

%\bibitem{OS} T. Oda and C. S. Seshadri,
%\emph{Compactifications of the generalized Jacobian variety.}
% Trans. Amer. Math. Soc. {\bf 253} (1979), 1-–90.

\bibitem{MW} S. Molcho, J. Wise, 
\emph{The logarithmic Picard group and its tropicalization},
arXiv:1807.11364.

\bibitem{Mum} D. Mumford, 
\emph{An analytic construction of degenerating abelian varieties over complete rings}, 
Compos. Math. {\bf 24} (1972), 239--272.

\bibitem{Pand} R. Pandharipande, 
\emph{A compactification over $\ol{M_g}$ of the universal moduli space of slope-semistable vector bundles.} 
J. Amer. Math. Soc.  \textbf{9} (1996), 425--471.

\bibitem{R1} D. Ranganathan,
\emph{Skeletons of stable maps I: rational curves in toric varieties.}
Journal of the London Mathematical Society {\bf 95} (3) (2017), 804--832.

\bibitem{R2} D. Ranganathan,
\emph{Skeletons of stable maps II: superabundant geometries.}
Research in the Mathematical Sciences {\bf 4} (11) (2017), 1--18.

%\bibitem{Shen} J. Shen,
%\emph{Break divisors and compactified Jacobians.}
%PhD Thesis, Yale University, 2016.

%\bibitem{T} A. Thuillier, 
%\emph{G\'eom\' etrie toroidale et g\' eom\'etrie analytique nonarchim\'edienne. Application au type d' homotopie de certains sch\'emas formels.} 
%Manuscripta Math. {\bf 123} {\bf 4} (2007), 381--451. 

\bibitem{U} M. Ulirsch,
\emph{Tropical geometry of moduli spaces of weighted stable curves.} 
 Journal of the London Mathematical Society {\bf 92} {\bf 2} (2015), 427--450.

%\bibitem{U17} M. Ulirsch,
%\emph{Tropicalization is a non-Archimedean analytic stack quotient}.
%Mathematical Research Letters {\bf 24} {\bf 4} (2017), 1205--1237.

%\bibitem{Viv13} F. Viviani, 
%\emph{Tropicalizing vs Compactifying the Torelli morphism.}  
%In: Tropical and Non-Archimedean Geometry, Contemp. Math. \textbf{605} (2013), 181--210.
\end{thebibliography}
\end{document}